\RequirePackage{fix-cm}
\documentclass[smallextended]{svjour3}       
\smartqed  
\usepackage{graphicx}
\usepackage{amssymb}
\usepackage{amsmath}
\usepackage{amsfonts}
\usepackage{amsmath}
\begin{document}

\title{Analysis of Deep Ritz Methods for Laplace Equations with Dirichlet Boundary Condition
}


\author{Chenguang Duan  \and
        Yuling Jiao \and
        Yanming Lai \and
        Xiliang Lu \and
        Qimeng Quan \and
        Jerry Zhijian Yang
}


\institute{Chenguang Duan \at
           School of Mathematics and Statistics, Wuhan University, Wuhan 430072, P.R. China. \\
           \email{cgduan.math@whu.edu.cn}
           \and
           Yuling Jiao \at
           School of Mathematics and Statistics, and Hubei Key Laboratory of Computational Science, Wuhan University, Wuhan 430072, P.R. China. \\
           \email{yulingjiaomath@whu.edu.cn}
           \and
           Yanming Lai \at
           School of Mathematics and Statistics, Wuhan University, Wuhan 430072, P.R. China. \\
           \email{laiyanming@whu.edu.cn}
           \and
           Xiliang Lu \at
           School of Mathematics and Statistics, and Hubei Key Laboratory of Computational Science, Wuhan University, Wuhan 430072, P.R. China. \\
           \email{xllv.math@whu.edu.cn}
           \and
           Qimeng Quan \at
           School of Mathematics and Statistics, Wuhan University, Wuhan 430072, P.R. China. \\
           \email{quanqm@whu.edu.cn}
           \and
           Jerry Zhijian Yang \at
           School of Mathematics and Statistics, and Hubei Key Laboratory of Computational Science, Wuhan University, Wuhan 430072, P.R. China. \\
           \email{zjyang.math@whu.edu.cn}
}

\date{Received: date / Accepted: date}

\maketitle

\begin{abstract}
\par Deep Ritz methods (DRM) have been proven numerically to be efficient in solving partial differential equations. In this paper, we present a convergence rate in $H^{1}$ norm for deep Ritz methods for Laplace equations with Dirichlet boundary condition, where the error depends on the depth and width in the deep neural networks and the number of samples explicitly. Further we can properly choose the depth and width in the deep neural networks in terms of the number of training samples. The main idea of the proof is to decompose the total error of DRM into three parts, that is approximation error, statistical error and the error caused by the boundary penalty. We bound the approximation error in $H^{1}$ norm with $\mathrm{ReLU}^{2}$ networks and control the statistical error via Rademacher complexity. In particular, we derive the bound on the Rademacher complexity of the non-Lipschitz composition of gradient norm with $\mathrm{ReLU}^{2}$ network, which is of immense independent  interest. We also analysis the error inducing by  the boundary penalty method and give a prior rule for tuning the penalty parameter.
\end{abstract}

\section{Introduction}
\label{intro}
\par Partial differential equations (PDEs) are one of the fundamental mathematical models in studying a variety of phenomenons arising in science and engineering. There have been established many conventional numerical methods successfully for solving PDEs in the case of low dimension $(d\leq3)$, particularly the finite element method \cite{brenner2007mathematical,ciarlet2002finite,Quarteroni2008Numerical,Thomas2013Numerical,Hughes2012the}. However,
one will encounter some  difficulties in both of theoretical analysis and numerical implementation when extending conventional numerical schemes  to high-dimensional PDEs.
The classic analysis of convergence, stability and any other properties will be trapped into troublesome situation due to the complex construction of finite element space \cite{ciarlet2002finite,brenner2007mathematical}. Moreover, in the term of practical computation, the scale of the discrete problem will increase exponentially with respect to the  dimension.

Motivated by the well-known fact that deep learning method for  high-dimensional data analysis has been achieved great successful applications in discriminative, generative and reinforcement learning \cite{he2015delving,Goodfellow2014Generative,silver2016mastering}, solving high dimensional PDEs with  deep neural networks becomes an extremely potential  approach and  has  attracted much attentions  \cite{Cosmin2019Artificial,Justin2018DGM,DeepXDE,raissi2019physics,Weinan2017The,Yaohua2020weak,Berner2020Numerically,Han2018solving}. Roughly speaking, these works can be divided into three categories. The first category is using deep neural network to improve  classical numerical methods, see for example  \cite{Kiwon2020Solver,Yufei2020Learning,hsieh2018learning,Greenfeld2019Learning}.
In the second category, the neural operator is introduced to learn mappings between infinite-dimensional spaces with neural networks \cite{Li2020Advances,anandkumar2020neural,li2021fourier}.
For the last category, one utilizes deep neural networks to approximate the solutions of PDEs directly including physics-informed neural networks (PINNs) \cite{raissi2019physics}, deep Ritz method (DRM)  \cite{Weinan2017The} and weak adversarial networks (WAN) \cite{Yaohua2020weak}. PINNs  is based on residual minimization for solving PDEs  \cite{Cosmin2019Artificial,Justin2018DGM,DeepXDE,raissi2019physics}.
Proceed from the variational form, \cite{Weinan2017The,Yaohua2020weak,Xu2020finite} propose neural-network based methods related to classical Ritz and Galerkin method. In \cite{Yaohua2020weak}, WAN are proposed inspired by Galerkin method. Based on Ritz method, \cite{Weinan2017The} proposes the DRM to solve variational problems corresponding to a class of PDEs.

\subsection{Related works and contributions}
\par
The idea using neural networks to solve PDEs goes back to 1990's \cite{Isaac1998Artificial,Dissanayake1994neural}.
Although there are great empirical achievements in recent several years, a challenging and interesting  question is to provide a rigorous error analysis such as finite element method.
Several recent efforts  have been devoted to making processes along this line,  see for example  \cite{e2020observations,Luo2020TwoLayerNN,Mishra2020EstimatesOT,Mller2021ErrorEF,lu2021priori,hong2021rademacher,Shin2020ErrorEO,Wang2020WhenAW,e2021barron}.
In \cite{Luo2020TwoLayerNN}, least squares minimization method with two-layer neural networks is studied, the optimization error under the assumption of over-parametrization and generalization error without the over-parametrization assumption are analyzed.
In \cite{lu2021priori,Xu2020finite}, the generalization error bounds of two-layer neural networks are derived via assuming that the exact solutions lie in spectral Barron space.

Dirichlet boundary condition  corresponding to a constrained minimization problem, which may cause some difficulties in computation. The penalty method has been applied in finite element methods and finite volume method \cite{Babuska1973The,Maury2009Numerical}. It is also been used in deep PDEs solvers \cite{Weinan2017The,raissi2019physics,Xu2020finite} since it is not easy to construct a network with given values on the boundary. We also apply penalty method to DRM with $\mathrm{ReLU}^2$ activation functions, and obtain the error estimation in this work. The main contribution are listed as follows:
\begin{itemize}
	\item
	We derive  a bound on  the approximation error of deep $\mathrm{ReLU}^2$ network in $H^1$ norm, which is of independent interest, see Theorem \ref{thm:apperr}. That is, for any $u_{\lambda}^*\in H^2(\Omega)$, there exist a $\mathrm{ReLU}^2$ network $\bar{u}_{\bar{\phi}}$ with depth $\mathcal{D} \leq \lceil\log_2d\rceil+3,$  width $ \mathcal{W}\leq \mathcal{O}\left(4d\left \lceil \frac{1}{\epsilon}-4\right \rceil^d\right)$ (where $d$ is the dimension),
	such that
	\begin{equation*}
		\left\|u_{\lambda}^*-\bar{u}_{\bar{\phi}}\right\|^2_{H^{1}\left(\Omega\right)}\leq\epsilon^2 \quad  \mbox{and} \quad \left\|Tu_{\lambda}^*-T\bar{u}_{\bar{\phi}}\right\|^2_{L^{2}\left(\partial\Omega\right)}\leq C_{d}\epsilon^{2}.
	\end{equation*}
	\item
	We establish a  bound on  the statistical error in  DRM with the tools of pseudo-dimension, especially we give a bound on
	\begin{equation*}
		\mathbb{E}_{Z_i,\sigma_i, i=1,...,n}\left[\sup_{u_{\phi}\in \mathcal{N}^2}\frac{1}{n}\left|\sum_{i} \sigma_i \left\|\nabla u_{\phi}(Z_i)\right\|^2\right|\right],
	\end{equation*}
	i.e., the Rademacher complexity  of the non-Lipschitz composition of  gradient norm  and   $\mathrm{ReLU}^2$ network,   via  calculating  the
	Pseudo dimension of networks with both $\mathrm{ReLU}$ and  $\mathrm{ReLU}^2$ activation functions, see Theorem \ref{thm:staerr}. The technique we used here is also helpful for bounding the statistical errors to other deep PDEs solvers.
	\item
	We give an upper bound of the error caused by the Robin approximation without additional assumptions, i.e.,  bound the error between the minimizer of the penalized form $u_{\lambda}^{*}$ and the weak solutions of the Laplace equation $u^{*}$, see Theorem \ref{thm:penaltyerr},
	\begin{equation*}
		\|u_{\lambda}^{*}-u^{*}\|_{H^{1}(\Omega)}\leq \mathcal{O}(\lambda^{-1}).
	\end{equation*}
	This result improves the one established in \cite{Mller2021ErrorEF,muller2020deep,hong2021rademacher}.
	\item Based on the above two error bounds we establish a nonasymptotic convergence rate of deep Ritz method for Laplace equation with Dirichlet boundary condition.
	We prove that if we set
	\begin{equation*}
		\mathcal{D} \leq \lceil\log_2d\rceil+3, \ \mathcal{W}\leq \mathcal{O}\left(4d\left \lceil \left(\frac{n}{\log n}\right)^{\frac{1}{2(d+2)}}-4\right \rceil^d \right),
	\end{equation*}
	and
	\begin{equation*}
		\lambda \sim n^{\frac{1}{3(d+2)}}(\log n)^{-\frac{d+3}{3(d+2)}},
	\end{equation*}
	it holds that
	\begin{equation*}
		\mathbb{E}_{\boldsymbol{X},\boldsymbol{Y}}\left[\|\widehat{u}_{\phi}-u^{*}\|_{H^{1}(\Omega)}^2\right]
		\leq \mathcal{O}\left(n^{-\frac{2}{3(d+2)}}\log n\right).
	\end{equation*}
	where $n$ is the number of training samples on  both the domain and the boundary.
	Our theory shed lights  on how to  choose the topological structure of the employed  networks and tune the penalty parameters  to achieve the desired convergence rate in terms of number of training samples.
\end{itemize}

\par Recently, \cite{Mller2021ErrorEF,muller2020deep} also study the convergence of DRM with Dirichlet boundary condition via penalty method.
However, the results of derive in \cite{Mller2021ErrorEF,muller2020deep} are quite different from ours.
Firstly, the approximation results in \cite{Mller2021ErrorEF,muller2020deep} in based on the  approximation error  of $\mathrm{ReLU}$ networks in Sobolev norms established in \cite{guhring2019error}. However, the $\mathrm{ReLU}$ network  may not be suitable for solving PDEs. In this work, we derive an upper bound on the approximation error  of $\mathrm{ReLU}^2$ networks  in $H^1$ norm, which is of independent interest.
Secondly, to analyze the error caused by the penalty term,  \cite{Mller2021ErrorEF,muller2020deep} assumed some additional conditions, and we do not need these conditions to obtain the error inducing by the penalty. Lastly,  we provide the convergence rate analysis involving the statistical error caused by finite samples used in the SGD training,  while in \cite{Mller2021ErrorEF,muller2020deep} they  do  not consider the statistical error at all. Moreover, to bound the statistical error  we need to control the Rademacher complexity  of the non-Lipschitz composition of  gradient norm  and   $\mathrm{ReLU}^2$ network, such technique can be useful for bounding the statistical errors to other deep PDEs solvers.

\par The rest of this paper is organized as follows.
In Section \ref{section:Preliminaries} we describe briefly the model problem and recall some standard properties of PDEs and variational problems. We also introduce some notations in deep Ritz methods as preliminaries.
We devote Section \ref{section:error} to the detail analysis on the convergence rate of the deep Ritz method with penalty, where various error estimations are analyzed rigorously one by one and the main results on the convergence rate are presented.
Some concluding remarks and discussions are given in Section \ref{section:conclusion}.

\section{Preliminaries}\label{section:Preliminaries}
\par Consider the following elliptic equation with zero-boundary condition
\begin{equation}\label{eq:dirichlet}
	\left\{ \begin{aligned}
		-\Delta u+wu & =f &  & \text{in}\ \Omega          \\
		u            & =0 &  & \text{on}\ \partial\Omega, \\
	\end{aligned}\right.
\end{equation}
where $\Omega$ is a bounded open subset of $\mathbb{R}^{d}$, $d>1$,  $f\in L^2(\Omega)$ and $w\in L^{\infty}(\Omega)$. Moreover, we suppose the coefficient $w$ satisfies $w\geq c_{1}\geq0$ a.e.. Without loss of generality, we assume $\Omega = [0,1]^{d}$. Define the bilinear form
\begin{equation}\label{eq:bilinear}
	a:H^1(\Omega)\times H^1(\Omega)\to\mathbb{R}, \quad (u,v)\mapsto\int_{\Omega} \nabla u\cdot\nabla v+ wuv \ \mathrm{d}x,
\end{equation}
and the corresponding quadratic energy functional by
\begin{equation}\label{eq:energy}
	\mathcal{L}(u) = \frac{1}{2}a(u,u)-\langle f,u\rangle_{L^2(\Omega)}= \frac{1}{2}|u|_{H^{1}(\Omega)}^{2} + \frac{1}{2}\|u\|_{L^{2}(\Omega;{w})}^{2}- \langle f,u \rangle_{L^2(\Omega)}.
\end{equation}

\begin{lemma}\cite{Evans2010PartialDE}
	The unique weak solution $u^* \in H_0^{1}(\Omega)$ of \eqref{eq:dirichlet} is the unique minimizer of $\mathcal{L}(u)$ over $H_{0}^{1}(\Omega)$. Moreover, $u^*\in H^{2}(\Omega)$.
\end{lemma}

\par Now we introduce the Robin approximation of \eqref{eq:dirichlet} with $\lambda>0$ as below
\begin{equation}\label{eq:robin}
	\left\{ \begin{aligned}
		-\Delta u+wu                                     & =f &  & \text{in}\ \Omega          \\
		\frac{1}{\lambda}\frac{\partial u}{\partial n}+u & =0 &  & \text{on}\ \partial\Omega. \\
	\end{aligned}\right.
\end{equation}
Similarly, we define the bilinear form
\begin{equation*}
	a_{\lambda}:H^1(\Omega)\times H^1(\Omega)\to\mathbb{R}, \quad (u,v)\mapsto a(u,v)+\lambda\int_{\partial\Omega}uv\ \mathrm{d}s,
\end{equation*}
and the corresponding quadratic energy functional with boundary penalty
\begin{equation}\label{eq:energy:penalty}
	\mathcal{L}_{\lambda}(u)= \frac{1}{2}a_{\lambda}(u,u)-\langle f,u \rangle_{L^2(\Omega)} = \mathcal{L}(u)+\frac{\lambda}{2}||Tu||_{L^{2}(\partial\Omega)}^{2},
\end{equation}
where $T$ means the trace operator.

\begin{lemma}\label{lemma:regularity}
	The unique weak solution $u_{\lambda}^* \in H^{1}(\Omega)$ of \eqref{eq:robin} is the unique minimizer of $\mathcal{L}_{\lambda}(u)$ over $H^{1}(\Omega)$. Moreover, $u_{\lambda}^{*}\in H^{2}(\Omega)$.
\end{lemma}
\begin{proof}
	See Appendix \ref{app:lemma:regularity}.
\end{proof}
From the perspective of infinite dimensional optimization, $\mathcal{L}_{\lambda}$ can be seen as the penalized version of $\mathcal{L}$. The following lemma provides the relationship between the minimizers of them.
\begin{lemma}\label{lemma:lambda_convergence}
	The minimizer $u_{\lambda}^{*}$ of the penalized problem \eqref{eq:energy:penalty} converges to $u^{*}$ in $H^1(\Omega)$ as $\lambda\rightarrow\infty$.
\end{lemma}

\begin{proof}
	This result follows from Proposition 2.1 in \cite{Maury2009Numerical} directly.
\end{proof}

\par The deep Ritz method can be divided into three steps. First, one use deep neural network to approximate the trial function. A deep neural network $u_{\phi}:\mathbb{R}\rightarrow\mathbb{R}^{N_{L}}$ is defined by
\begin{equation*}
	\begin{aligned}
		&u_{0}(\boldsymbol{x})=\boldsymbol{x}, \\
		&u_{\ell}(\boldsymbol{x})=\sigma_{\ell}(A_{\ell}u_{\ell-1}+b_{\ell}), \quad \ell=1,2,\ldots,L-1, \\
		&u=u_{L}(\boldsymbol{x})=A_{L}u_{L-1}+b_{L},
	\end{aligned}
\end{equation*}
where $A_{\ell}\in\mathbb{R}^{N_{\ell}\times N_{\ell-1}}$, $b_{\ell}\in\mathbb{R}^{N_{\ell}}$ and the activation functions $\sigma_{\ell}$ may be different for different $\ell$. The depth $\mathcal{D}$ and the width $\mathcal{W}$ of neural networks $u_{\phi}$ are defined as
\begin{equation*}
	\mathcal{D}=L, \ \mathcal{W}=\max\{N_{\ell}:\ell=1,2,\ldots,L\}.
\end{equation*}
$\sum_{\ell=1}^{L}N_{\ell}$ is called the number of units of $u_{\phi}$, and $\phi=\{A_{\ell},b_{\ell}\}_{\ell=1}^{N}$ is called the free parameters of the network.
\begin{definition}
	The class $\mathcal{N}_{\mathcal{D},\mathcal{W},\mathcal{B}}^{\alpha}$ is the collection of neural networks $u_{\phi}$ which satisfies that
	\begin{itemize}
		\item[(i)] depth and width are $\mathcal{D}$ and $\mathcal{W}$, respectively;
		\item[(ii)] the function values $u_{\phi}(\boldsymbol{x})$ and the squared norm of $\nabla u_{\phi}(\boldsymbol{x})$ are bounded by $\mathcal{B}$;
		\item[(iii)] activation functions are given by $\mathrm{ReLU}^{\alpha}$, where $\alpha$ is the (multi-)index.
	\end{itemize}
\end{definition}
For example, $\mathcal{N}_{\mathcal{D},\mathcal{W},\mathcal{B}}^{2}$ is the class of networks with activation functions as $\mathrm{ReLU}^2$, and $\mathcal{N}_{\mathcal{D},\mathcal{W},\mathcal{B}}^{1,2}$ is that with activation functions as $\mathrm{ReLU}^1$ and $\mathrm{ReLU}^2$. We may simply use $\mathcal{N}^\alpha$ if there is no confusion.

\par Second, one use Monte Carlo method to discretize the energy functional. We rewrite \eqref{eq:energy:penalty} as
\begin{equation}\label{eq:energy:penalty2}
	\begin{aligned}
		\mathcal{L}_{\lambda}(u) = & |\Omega|\mathop{\mathbb{E}}_{X\sim U(\Omega)}\left[ \frac{\|\nabla u(X)\|_{2}^{2}}{2}+\frac{w(X)u^{2}(X)}{2}-u(X)f(X) \right] \\
		&+\frac{\lambda}{2}|\partial\Omega|\mathop{\mathbb{E}}_{Y\sim U(\partial\Omega)}\left[ Tu^{2}(Y) \right],
	\end{aligned}
\end{equation}
where $U(\Omega)$, $U(\partial\Omega)$ are the uniform distribution on $\Omega$ and $\partial\Omega$.
We now introduce the discrete version of \eqref{eq:energy:penalty} and replace $u$ by neural network $u_{\phi}$, as follows
\begin{equation}\label{eq:energy:penalty:discrete}
	\begin{aligned}
		\widehat{\mathcal{L}}_{\lambda}(u_{\phi}) = & \frac{|\Omega|}{N}\sum_{i=1}^{N}\left[ \frac{\|\nabla u_{\phi}(X_{i})\|_{2}^{2}}{2}+\frac{w(X_{i})u_{\phi}^{2}(X_{i})}{2}-u_{\phi}(X_{i})f(X_{i}) \right] \\
		& +\frac{\lambda}{2}\frac{|\partial\Omega|}{M}\sum_{j=1}^{M}\left[ Tu_{\phi}^{2}(Y_{j}) \right].
	\end{aligned}
\end{equation}
We denote the minimizer of \eqref{eq:energy:penalty:discrete} over $\mathcal{N}^{2}$ as $\widehat{u}_{\phi}$, that is
\begin{equation}\label{eq:u_hat_star}
	\widehat{u}_{\phi}=\mathop{\arg\min}_{u_{\phi}\in \mathcal{N}^{2}}\widehat{\mathcal{L}}_{\lambda}(u_{\phi}),
\end{equation}
where $\{X_{i}\}_{i=1}^{N}\sim U(\Omega)$ i.i.d. and $\{Y_{j}\}_{j=1}^{M}\sim U(\partial\Omega)$ i.i.d..
\par  Finally, we choose an algorithm for solving the optimization problem, and denote $u_{\phi_{\mathcal{A}}}$ as the solution by optimizer $\mathcal{A}$.

\section{Error Analysis}\label{section:error}
\par In this section we prove the convergence rate analysis for DRM with deep $\mathrm{ReLU}^{2}$ networks. The following Theorem plays an important role by decoupling the total errors into four types of errors.
\begin{theorem}\label{thm:errdec}
	\begin{equation*}
		\begin{aligned}
			& \| u_{\phi_\mathcal{A}}-u^{*} \|_{H^{1}(\Omega)}^{2} \\
			\leq & \frac{4}{c_{1}\wedge 1}\left\{\underbrace{\inf_{\bar{u}\in\mathcal{N}^2}\left[\frac{\|w\|_{L^{\infty}(\Omega)} \vee 1}{2} \|\bar{u}-u_{\lambda}^{*}\|_{H^{1}(\Omega)}^{2}+\frac{\lambda}{2}\|T\bar{u}-Tu_{\lambda}^{*} \|_{L^2({\partial\Omega})}^{2}\right]}_{\mathcal{E}_{app}}\right.  \\
			& \left.+\underbrace{2 \sup _{u \in \mathcal{N}^2}\left|\mathcal{L}_{\lambda}(u)-\widehat{\mathcal{L}}_{\lambda}(u)\right|}_{\mathcal{E}_{sta}}+\underbrace{\left[\widehat{\mathcal{L}}_{\lambda}\left(u_{\phi_{\mathcal{A}}}\right)-\widehat{\mathcal{L}}_{\lambda}\left(\widehat{u}_{\phi}\right)\right]}_{\mathcal{E}_{opt}}\right\} + 2 \underbrace{\| u_{\lambda}^{*}-u^{*}\|_{H^{1}(\Omega)}^{2}}_{\mathcal{E}_{pen}}
		\end{aligned}
	\end{equation*}
\end{theorem}
\begin{proof}
	\par Given $u_{\phi_\mathcal{A}}\in H^{1}(\Omega)$, we can decompose its distance to the weak solution of \eqref{eq:dirichlet} using triangle inequality
	\begin{equation}\label{eq:errdec}
		\| u_{\phi_\mathcal{A}}-u^{*} \|_{H^{1}(\Omega)}\leq \| u_{\phi_\mathcal{A}}-u_{\lambda}^{*} \|_{H^{1}(\Omega)}+\| u_{\lambda}^{*}-u^{*} \|_{H^{1}(\Omega)}.
	\end{equation}
	First, we decouple the first term into three parts. For any $\bar{u}\in\mathcal{N}^2$, we have
	\begin{equation*}
		\begin{aligned}
			& \mathcal{L}_{\lambda}\left(u_{\phi_{\mathcal{A}}}\right)-\mathcal{L}_{\lambda}\left(u_{\lambda}^{*}\right)\\
			& =\mathcal{L}_{\lambda}\left(u_{\phi_{\mathcal{A}}}\right)-\widehat{\mathcal{L}}_{\lambda}\left(u_{\phi_{\mathcal{A}}}\right) +\widehat{\mathcal{L}}_{\lambda}\left(u_{\phi_{\mathcal{A}}}\right)-\widehat{\mathcal{L}}_{\lambda}\left(\widehat{u}_{\phi}\right) +\widehat{\mathcal{L}}_{\lambda}\left(\widehat{u}_{\phi}\right)-\widehat{\mathcal{L}}_{\lambda}\left(\bar{u}\right) \\
			& \quad+\widehat{\mathcal{L}}_{\lambda}\left(\bar{u}\right)-\mathcal{L}_{\lambda}\left(\bar{u}\right)+\mathcal{L}_{\lambda}\left(\bar{u}\right)-\mathcal{L}_{\lambda}\left(u_{\lambda}^{*}\right)\\
			& \leq \left[\mathcal{L}_{\lambda}\left(\bar{u}\right)-\mathcal{L}_{\lambda}\left(u_{\lambda}^{*}\right)\right]+2 \sup _{u \in \mathcal{N}^2}\left|\mathcal{L}_{\lambda}(u)-\widehat{\mathcal{L}}_{\lambda}(u)\right|+\left[\widehat{\mathcal{L}}_{\lambda}\left(u_{\phi_{\mathcal{A}}}\right)-\widehat{\mathcal{L}}_{\lambda}\left(\widehat{u}_{\phi}\right)\right].
		\end{aligned}
	\end{equation*}
	Since $\bar{u}$ can be any element in $\mathcal{N}^2$, we take the infimum of $\bar{u}$
	\begin{equation}\label{eq:errdec1}
		\begin{aligned}
			\mathcal{L}_{\lambda}\left(u_{\phi_{\mathcal{A}}}\right)-\mathcal{L}_{\lambda}\left(u_{\lambda}^{*}\right)
			\leq & \inf_{\bar{u}\in\mathcal{N}^2}\left[\mathcal{L}_{\lambda}\left(\bar{u}\right)-\mathcal{L}_{\lambda}\left(u^{*}\right)\right]+2 \sup _{u \in \mathcal{N}^2}\left|\mathcal{L}_{\lambda}(u)-\widehat{\mathcal{L}}_{\lambda}(u)\right| \\ &+\left[\widehat{\mathcal{L}}_{\lambda}\left(u_{\phi_{\mathcal{A}}}\right)-\widehat{\mathcal{L}}_{\lambda}\left(\widehat{u}_{\phi}\right)\right].
		\end{aligned}
	\end{equation}
	For any $u\in\mathcal{N}$, set $v=u-u_{\lambda}^{*}$, then
	\begin{equation*}
		\begin{aligned}
			\mathcal{L}_{\lambda}(u)= & \mathcal{L}_{\lambda}(u_{\lambda}^{*}+v)\\
			= & \frac{1}{2}\langle \nabla(u_{\lambda}^{*}+v),\nabla(u_{\lambda}^{*}+v) \rangle_{L^2({\Omega})}+\frac{1}{2}\langle u_{\lambda}^{*}+v,u_{\lambda}^{*}+v \rangle_{L^2({\Omega};w)}-\langle u_{\lambda}^{*}+v,f \rangle_{L^2({\Omega})} \\
			& +\frac{\lambda}{2}\langle T(u_{\lambda}^{*}+v),T(u_{\lambda}^{*}+v) \rangle_{L^2({\partial\Omega})} \\
			= & \mathcal{L}_{\lambda}(u_{\lambda}^{*})+\langle \nabla u_{\lambda}^{*},\nabla v\rangle_{L^2({\Omega})}+\langle u_{\lambda}^{*},v \rangle_{L^2({\Omega;w})}-\langle v,f \rangle_{L^2({\Omega})}+\lambda\langle Tu_{\lambda}^{*},Tv \rangle_{L^2({\partial\Omega})} \\
			& +\frac{1}{2}\langle \nabla v,\nabla v\rangle_{L^2({\Omega})}+\frac{1}{2}\langle v,v \rangle_{L^2({\Omega};w)} +\frac{\lambda}{2}\langle Tv,Tv \rangle_{L^2({\partial\Omega})}                                  \\
			= & \mathcal{L}_{\lambda}(u_{\lambda}^{*})+\frac{1}{2}\langle \nabla v,\nabla v\rangle_{L^2({\Omega})}+\frac{1}{2}\langle v,v \rangle_{L^2({\Omega};w)}+\frac{\lambda}{2}\langle Tv,Tv \rangle_{L^2({\partial\Omega})},
		\end{aligned}
	\end{equation*}
	where the last equality comes from the fact that $u_{\lambda}^{*}$ is the minimizer of \eqref{eq:energy:penalty}. Therefore
	\begin{equation*}
		\mathcal{L}_{\lambda}(u)-\mathcal{L}_{\lambda}(u_{\lambda}^{*})=\frac{1}{2}\langle \nabla v,\nabla v\rangle_{L^2({\Omega})}+\frac{1}{2}\langle v,v \rangle_{L^2({\Omega};w)}+\frac{\lambda}{2}\langle Tv,Tv \rangle_{L^2({\partial\Omega})},
	\end{equation*}
	that is
	\begin{equation}\label{eq:errdec2}
		\begin{aligned}
			\frac{c_{1}\wedge 1}{2} \|u-u_{\lambda}^{*}\|_{H^{1}(\Omega)}^{2} & \leq\mathcal{L}_{\lambda}(u)-\mathcal{L}_{\lambda}(u_{\lambda}^{*})-\frac{\lambda}{2}\|Tu-Tu_{\lambda}^{*} \|_{L^2({\partial\Omega})}^{2} \\
			& \leq\frac{\|w\|_{L^{\infty}(\Omega)}\vee1}{2} \|u-u_{\lambda}^{*}\|_{H^{1}(\Omega)}^{2}.
		\end{aligned}
	\end{equation}
	Combining \eqref{eq:errdec1} and \eqref{eq:errdec2}, we obtain
	\begin{equation}\label{eq:errdec3}
		\begin{aligned}
			& \|u_{\phi_{\mathcal{A}}}-u_{\lambda}^{*}\|_{H^{1}(\Omega)}^{2}                                                                                                                                                                                                                    \\
			\leq & \frac{2}{c_{1}\wedge 1}\left\{\mathcal{L}_{\lambda}(u_{\phi_{\mathcal{A}}})-\mathcal{L}_{\lambda}(u_{\lambda}^{*})-\frac{\lambda}{2}\|Tu_{\phi_{\mathcal{A}}}-Tu_{\lambda}^{*} \|_{L^2({\partial\Omega})}^{2}\right\}                                                             \\
			\leq & \frac{2}{c_{1}\wedge 1}\left\{\inf_{\bar{u}\in\mathcal{N}^2}\left[\mathcal{L}_{\lambda}\left(\bar{u}\right)-\mathcal{L}_{\lambda}\left(u_{\lambda}^{*}\right)\right]+2 \sup _{u \in \mathcal{N}^2}\left|\mathcal{L}_{\lambda}(u)-\widehat{\mathcal{L}}_{\lambda}(u)\right|\right. \\ &\left.+\left[\widehat{\mathcal{L}}_{\lambda}\left(u_{\phi_{\mathcal{A}}}\right)-\widehat{\mathcal{L}}_{\lambda}\left(\widehat{u}_{\phi}\right)\right]\right\} \\
			\leq & \frac{2}{c_{1}\wedge 1}\left\{\inf_{\bar{u}\in\mathcal{N}^2}\left[\frac{\|w\|_{L^{\infty}(\Omega)} \vee 1}{2} \|\bar{u}-u_{\lambda}^{*}\|_{H^{1}(\Omega)}^{2}+\frac{\lambda}{2}\|T\bar{u}-Tu_{\lambda}^{*} \|_{L^2({\partial\Omega})}^{2}\right]\right.                           \\
			& \left.+2 \sup _{u \in \mathcal{N}^2}\left|\mathcal{L}_{\lambda}(u)-\widehat{\mathcal{L}}_{\lambda}(u)\right|+\left[\widehat{\mathcal{L}}_{\lambda}\left(u_{\phi_{\mathcal{A}}}\right)-\widehat{\mathcal{L}}_{\lambda}\left(\widehat{u}_{\phi}\right)\right]\right\}.
		\end{aligned}
	\end{equation}
	Substituting \eqref{eq:errdec3} into \eqref{eq:errdec}, it is evident to see that the theorem holds.
\end{proof}
\par The approximation error $\mathcal{E}_{app}$ describes the expressive power of the $\mathrm{ReLU}^2$ networks $\mathcal{N}^{2}$ in $H^{1}$ norm, which corresponds to the approximation error in FEM known as the C\'{e}a's lemma \cite{ciarlet2002finite}. The statistical error $\mathcal{E}_{sta}$ is caused by the Monte Carlo discritization of $L_{\lambda}(\cdot)$ defined in \eqref{eq:energy:penalty} with $\widehat{L}_{\lambda}(\cdot)$ in \eqref{eq:energy:penalty:discrete}. While, the optimization error $\mathcal{E}_{opt}$ indicates the performance of the solver $\mathcal{A}$ we utilized. In contrast, this error is corresponding to the error of solving linear systems in FEM. In this paper we consider the scenario of perfect training with $\mathcal{E}_{opt}$ = 0. The error $\mathcal{E}_{pen}$ caused by the boundary penalty is the distance between the minimizer of the energy with zero boundary condition and the minimizer of the energy with penalty.

\subsection{Approximation error}
\begin{theorem}\label{thm:apperr}
	Assume $\|u_{\lambda}^{*}\|_{H^{2}(\Omega)}\leq c_{2}$, then there exist a $\mathrm{ReLU}^2$ network $\bar{u}_{\bar{\phi}}\in\mathcal{N}^{2}$ with depth and width satisfying
	\begin{equation*}
		\mathcal{D}\leq \lceil\log_{2} d\rceil+3, \quad \mathcal{W}\leq 4d \left\lceil \frac{Cc_{2}}{\varepsilon}-4 \right\rceil^{d}
	\end{equation*}
	such that
	\begin{equation*}
		\begin{aligned}
			\mathcal{E}_{app} & =\inf_{\bar{u}\in\mathcal{N}^2}\left[\frac{\|w\|_{L^{\infty}(\Omega)} \vee 1}{2}\|\bar{u}-u_{\lambda}^{*}\|_{H^{1}(\Omega)}^{2}+\frac{\lambda}{2}\|T\bar{u}-Tu_{\lambda}^{*} \|_{L^2({\partial\Omega})}^{2}\right] \\
			& \leq \left( \frac{\|w\|_{L^{\infty}(\Omega)} \vee 1}{2}+\frac{\lambda C_{d}}{2} \right) \varepsilon ^{2}
		\end{aligned}
	\end{equation*}
	where $C$ is a genetic constant and $C_{d}>0$ is a constant depending only on $\Omega$.
\end{theorem}
\begin{proof}
	Our proof is based on some classical  approximation results of  B-splines \cite{schumaker2007spline,de1978practical}.
	Let us recall some notation and useful results.
	We denote by $\pi_{l}$ the dyadic partition of $[0,1]$, i.e.,
	\begin{equation*}
		\pi_{l}: t_{0}^{(l)}=0<t_{1}^{(l)}<\cdots<t_{2^{l}-1}^{(l)}<t_{2^{l}}^{(l)}=1,
	\end{equation*}
	where $t_{i}^{(l)}=i \cdot 2^{-l}(0\leq i\leq 2^l)$. The cardinal B-spline of order $3$ with respect to partition $\pi_l$ is defined by
	\begin{equation*}
		N_{l, i}^{(3)}(x)=(-1)^k\left[t_{i}^{(l)}, \ldots, t_{i+3}^{(l)},(x-t)_{+}^{2}\right] \cdot\left(t_{i+3}^{(l)}-t_{i}^{(l)}\right),\quad i=-2,\cdots,2^l-1
	\end{equation*}
	which can be rewritten in the following equivalent form,
	\begin{equation} \label{B-spline expression}
		N_{l, i}^{(3)}(x)=2^{2l-1}\sum_{j=0}^{3}(-1)^{j}\left(\begin{array}{l}
			3 \\
			j
		\end{array}\right)(x-i2^{-l}-j2^{-l})_{+}^{2},\quad i=-2,\cdots,2^l-1.
	\end{equation}
	The multivariate cardinal B-spline of order $3$ is defined by the product of univariate cardinal B-splines of order $3$, i.e.,
	\begin{equation*}
		{N}_{l, \boldsymbol{i}}^{(3)}(\boldsymbol{x})=\prod_{j=1}^{d} N_{l, i_j}^{(3)}\left(x_{j}\right), \quad \boldsymbol{i}=\left(i_{1}, \ldots, i_{d}\right),-3<i_{j}<2^{l}.
	\end{equation*}
	Denote
	\begin{equation*}
		S_{l}^{(3)}([0,1]^d)=\text{span}\{N_{l,\boldsymbol{i}}^{(3)},-3<i_j<2^l,j=1,2,\cdots,d\}.
	\end{equation*}
	Then, the element $f$ in $S_{l}^{(3)}([0,1]^d)$ are    piecewise  polynomial functions according to to partition $\pi_l^d$   with each piece being  degree $2$ and in $C^{1}([0,1]^d)$. Since
	\begin{equation*}
		S_1^{(3)}\subset S_2^{(3)}\subset S_3^{(3)}\subset\cdots,
	\end{equation*}
	We can further denote
	\begin{equation*}
		S^{(3)}([0,1]^d)=\bigcup_{l=1}^{\infty} S_{l}^{(3)}([0,1]^d).
	\end{equation*}
	The following approximation result  of cardinal B-splines in Sobolev spaces which is a direct consequence of theorem 3.4 in  \cite{schultz1969approximation}  play an important role in the  proof of this Theorem.
	\begin{lemma}\label{spapp}
		Assume  $u^*\in H^2([0,1]^d)$, there exists $\{c_j\}_{j=1}^{(2^l-4)^d}\subset\mathbb{R}$ with $l>2$ such that
		\begin{equation*}
			\|u^*-\sum_{j=1}^{(2^l-4)^d}c_j{N}_{l, \boldsymbol{i}_j}^{(3)}\|_{H^{1}(\Omega)}\leq \frac{C}{2^l}\|u^*\|_{H^1(\Omega)},
		\end{equation*}
		where $C$ is a constant only depend on $d$.
	\end{lemma}
	
	\begin{lemma}\label{spasrelu2}
		The multivariate B-spline ${N}_{l, \boldsymbol{i}}^{(3)}(\boldsymbol{x})$ can be implemented exactly by a $\mathrm{ReLU}^2$  network with depth  $\lceil\log_2d\rceil+2$ and width $4d$.
	\end{lemma}
	\begin{proof}
		Denote
		\begin{equation*}
			\sigma (x)=\left\{\begin{array}{ll}
				x^2, & x\geq 0 \\
				0, & \text{else}
			\end{array}\right.
		\end{equation*}
		as the activation function in $\mathrm{ReLU}^2$  network.
		By definition of $N_{l, i}^{(3)}(x)$ in (\ref{B-spline expression}),  it's clear that $N_{l, i}^{(3)}(x)$ can be implemented by  $\mathrm{ReLU}^2$  network  without any error with depth  $2$ and width $4$. On the other hand  $\mathrm{ReLU}^2$  network  can also realize  multiplication without any error. In fact, for any $x,y\in\mathbb{R}$,
		\begin{equation*}
			xy = \frac{1}{4}[(x+y)^2-(x-y)^2]=\frac{1}{4}[\sigma (x+y)+\sigma (-x-y)-\sigma (x-y)-\sigma (y-x)].
		\end{equation*}
		Hence multivariate B-spline of order $3$  can be implemented by  $\mathrm{ReLU}^2$  network exactly with depth  $\lceil\log_2d\rceil+2$ and width $4d$.
	\end{proof}
	
	For any $\epsilon>0$, by Lemma \ref{spapp} and \ref{spasrelu2} with $\frac{1}{2^l} \leq \left \lceil \frac{C\|u^*\|_{H^2}}{\epsilon}\right \rceil$,
	there exists $\bar{u}_{\bar{\phi}} \in \mathcal{N}^2$, such that
	\begin{equation}\label{H1 b-spline}
		\left\|u_{\lambda}^*-\bar{u}_{\bar{\phi}} \right\|_{H^1(\Omega)} \leq \epsilon.
	\end{equation}
	By the trace theorem, we have
	\begin{equation}\label{H1 b-spline:trace}
		\|Tu_{\lambda}^{*}-T\bar{u}_{\bar{\phi}}\|_{L^2({\partial\Omega})}\leq C_{d}^{1/2}\left\|u_{\lambda}^*-\bar{u}_{\bar{\phi}} \right\|_{H^1(\Omega)}\leq C_{d}^{1/2}\epsilon,
	\end{equation}
	where $C_{d}>0$ is a constant depending only on $\Omega$. The depth $\mathcal{D}$ and width $\mathcal{W}$ of $\bar{u}_{\bar{\phi}}$ are satisfying   $\mathcal{D}\leq\lceil\log_2d\rceil+3$ and $\mathcal{W}\leq 4dn=4d\left \lceil   \frac{C\|u^*\|_{H^2}}{\epsilon}-4\right \rceil^d$, respectively. Combining \eqref{H1 b-spline} and \eqref{H1 b-spline:trace}, we arrive at the result.
\end{proof}

\subsection{Statistical error}
\par In this section, we bound the statistical error
\begin{equation*}
	\mathcal{E}_{sta}=2\sup_{u\in\mathcal{N}^{2}}\left| \mathcal{L}_{\lambda}(u)-\widehat{\mathcal{L}}_{\lambda}(u) \right|.
\end{equation*}
For simplicity of presentation, we use $c_{3}$ to denote the upper bound of $f$, $w$ and suppose $c_{3}\geq\mathcal{B}$, that is
\begin{equation*}
	\|f\|_{L^{\infty}(\Omega)} \vee \|w\|_{L^{\infty}(\Omega)} \vee \mathcal{B} \leq c_{3} < \infty.
\end{equation*}

\par First, we need to decompose the statistical error into four parts, and estimate each one.
\begin{lemma}\label{lemma:staerr:dec}
	\begin{equation*}
		\sup _{u \in \mathcal{N}^2}\left|\mathcal{L}_{\lambda}(u)-\widehat{\mathcal{L}}_{\lambda}(u)\right|\leq\sum_{j=1}^{3}\sup _{u \in \mathcal{N}^2}\left|\mathcal{L}_{\lambda,j}(u)-\widehat{\mathcal{L}}_{\lambda,j}(u)\right|+\frac{\lambda}{2}\sup _{u \in \mathcal{N}^2}\left|\mathcal{L}_{\lambda,4}(u)-\widehat{\mathcal{L}}_{\lambda,4}(u)\right|,
	\end{equation*}
	where
	\begin{equation*}
		\mathcal{L}_{\lambda,1}(u)=|\Omega|\mathop{\mathbb{E}}_{X\sim U(\Omega)}\left[ \frac{\|\nabla u(X)\|_{2}^{2}}{2}\right],\quad \widehat{\mathcal{L}}_{\lambda,1}(u) =\frac{|\Omega|}{N}\sum_{i=1}^{N}\left[ \frac{\|\nabla u(X_{i})\|_{2}^{2}}{2}\right],
	\end{equation*}
	\begin{equation*}
		\mathcal{L}_{\lambda,2}(u)=|\Omega|\mathop{\mathbb{E}}_{X\sim U(\Omega)}\left[ \frac{w(X)u^{2}(X)}{2} \right] , \quad \widehat{\mathcal{L}}_{\lambda,2}(u) =\frac{|\Omega|}{N}\sum_{i=1}^{N}\left[ \frac{w(X_{i})u^{2}(X_{i})}{2} \right],
	\end{equation*}
	\begin{equation*}
		\mathcal{L}_{\lambda,3}(u) = |\Omega|\mathop{\mathbb{E}}_{X\sim U(\Omega)}\left[u(X)f(X) \right] , \quad \widehat{\mathcal{L}}_{\lambda,3}(u) =\frac{|\Omega|}{N}\sum_{i=1}^{N}\left[u(X_{i})f(X_{i}) \right],
	\end{equation*}
	\begin{equation*}
		\mathcal{L}_{\lambda,4}(u)=|\partial\Omega|\mathop{\mathbb{E}}_{Y\sim U(\partial\Omega)}\left[ Tu^{2}(Y) \right], \quad \widehat{\mathcal{L}}_{\lambda,4}(u)=\frac{|\partial\Omega|}{M}\sum_{j=1}^{M}\left[ Tu^{2}(Y_{j}) \right].
	\end{equation*}
\end{lemma}
\begin{proof}
	It is easy to verified by triangle inequality.
\end{proof}

\par We use $\mu$ to denote $\mathrm{U}(\Omega)(\mathrm{U}(\partial \Omega)) .$ Given $n=N(M)$ i.i.d samples $\mathbf{Z}_{n}=\left\{Z_{i}\right\}_{i=1}^{n}$ from $\mu$, with $Z_{i}=X_{i}\left(Y_{i}\right) \sim \mu$, we need the following Rademacher complexity to measure the capacity of the given function class $\mathcal{N}$ restricted on $n$ random samples $\mathbf{Z}_{n}$.

\begin{definition}
The Rademacher complexity of a set $A \subseteq \mathrm{R}^{n}$ is defined as
$$
\mathfrak{R}(A)=\mathbb{E}_{\mathbf{Z}_{n}, \Sigma_{n}}\left[\sup _{a \in A} \frac{1}{n}\left|\sum_{i} \sigma_{i} a_{i}\right|\right]
$$
where $\Sigma_{n}=\left\{\sigma_{i}\right\}_{i=1}^{n}$ are $n$ i.i.d Rademacher variables with $\mathbb{P}\left(\sigma_{i}=1\right)=\mathbb{P}\left(\sigma_{i}=-1\right)=1/2.$ The Rademacher complexity of function class $\mathcal{N}$ associate with random sample $\mathbf{Z}_{n}$ is defined as
\begin{equation*}
	\mathfrak{R}(\mathcal{N})=\mathbb{E}_{\mathbf{Z}_{n}, \Sigma_{n}}\left[\sup _{u \in \mathcal{N}} \frac{1}{n} \left| \sum_{i} \sigma_{i} u\left(Z_{i}\right) \right|\right].
\end{equation*}
\end{definition}

\par For the sake of simplicity, we deal with last three terms first.
\begin{lemma}\label{lemma:staerr:liprad}
	Suppose that $\psi:\mathbb{R}^{d}\times\mathbb{R}\rightarrow\mathbb{R}$, $(x,y)\mapsto\psi(x,y)$ is $\ell$-Lipschitz continuous on $y$ for all $x$. Let $\mathcal{N}$ be classes of functions on $\Omega$ and $\psi\circ\mathcal{N}=\{\psi\circ u:x\mapsto\psi(x,u(x)) ,u\in\mathcal{N}\}$. Then
	\begin{equation*}
		\mathfrak{R}(\psi\circ\mathcal{N})\leq\ell \ \mathfrak{R}(\mathcal{N})
	\end{equation*}
\end{lemma}
\begin{proof}
	Corollary 3.17 in \cite{ledoux2013probability}.
\end{proof}

\begin{lemma}\label{lemma:staerr:Rad234}
	\begin{equation*}
		\begin{aligned}
			\mathbb{E}_{\boldsymbol{Z}_{n}}\left[ \sup_{u\in\mathcal{N}^{2}}\left| \mathcal{L}_{\lambda,2}(u)-\widehat{\mathcal{L}}_{\lambda,2}(u) \right| \right] & \leq c_{3}^{2}\mathfrak{R}(\mathcal{N}^{2}), \\
			\mathbb{E}_{\boldsymbol{Z}_{n}}\left[ \sup_{u\in\mathcal{N}^{2}}\left| \mathcal{L}_{\lambda,3}(u)-\widehat{\mathcal{L}}_{\lambda,3}(u) \right| \right] & \leq c_{3}\mathfrak{R}(\mathcal{N}^{2}),     \\
			\mathbb{E}_{\boldsymbol{Z}_{n}}\left[ \sup_{u\in\mathcal{N}^{2}}\left| \mathcal{L}_{\lambda,4}(u)-\widehat{\mathcal{L}}_{\lambda,4}(u) \right| \right] & \leq 2c_{3}\mathfrak{R}(\mathcal{N}^{2}). \\
		\end{aligned}
	\end{equation*}
\end{lemma}
\begin{proof}
	Suppose $|y|<c_{3}$. Define
	\begin{equation*}
		\psi_{2}(x,y)= \frac{w(x)y^{2}}{2},\ \psi_{3}(x,y)= f(x)y,\ \psi_{4}(x,y)= y^{2}.
	\end{equation*}
	According to the symmetrization method, we have
	\begin{equation}\label{eq:staerr:sym}
		\begin{aligned}
		\mathbb{E}_{\boldsymbol{Z}_{n}}\left[ \sup_{u\in\mathcal{N}^{2}}\left| \mathcal{L}_{\lambda,2}(u)-\widehat{\mathcal{L}}_{\lambda,2}(u) \right| \right] & \leq \mathfrak{R}(\psi_{2}\circ\mathcal{N}^{2}), \\
		\mathbb{E}_{\boldsymbol{Z}_{n}}\left[ \sup_{u\in\mathcal{N}^{2}}\left| \mathcal{L}_{\lambda,3}(u)-\widehat{\mathcal{L}}_{\lambda,3}(u) \right| \right] & \leq \mathfrak{R}(\psi_{3}\circ\mathcal{N}^{2}),     \\
		\mathbb{E}_{\boldsymbol{Z}_{n}}\left[ \sup_{u\in\mathcal{N}^{2}}\left| \mathcal{L}_{\lambda,4}(u)-\widehat{\mathcal{L}}_{\lambda,4}(u) \right| \right] & \leq \mathfrak{R}(\psi_{4}\circ\mathcal{N}^{2}). \\
		\end{aligned}
	\end{equation}

	The result will follow from Lemma \ref{lemma:staerr:liprad} and \eqref{eq:staerr:sym} directly, if we can show that $\psi_{i}(x,y)$, $i=2,3,4$ are $c_{3}^{2}$, $c_{3}$, $2c_{3}$ -Lipschitz continuous on $y$ for all $x$, respectively. For arbitrary $y_{1}$, $y_{2}$ with $|y_{i}|\leq c_{3}$, $i=1,2$
	\begin{equation*}
		|\psi_{2}(x,y_{1})-\psi_{2}(x,y_{2})|=\left| \frac{w(x)y_{1}^{2}}{2}-\frac{w(x)y_{2}^{2}}{2} \right|=\frac{|w(x)(y_{1}+y_{2})|}{2}|y_{1}-y_{2}|\leq c_{3}^{2}|y_{1}-y_{2}|,
	\end{equation*}
	\begin{equation*}
		|\psi_{3}(x,y_{1})-\psi_{3}(x,y_{2})|=\left| f(x)y_{1}-f(x)y_{2} \right|=|f(x)||y_{1}-y_{2}|\leq c_{3}|y_{1}-y_{2}|,
	\end{equation*}
	\begin{equation*}
		|\psi_{4}(x,y_{1})-\psi_{4}(x,y_{2})|=\left| y_{1}^{2}-y_{2}^{2} \right|=|y_{1}+y_{2}||y_{1}-y_{2}|\leq 2c_{3}|y_{1}-y_{2}|.
	\end{equation*}
\end{proof}

\par We now turn to the most difficult term in Lemma \ref{lemma:staerr:dec}. Since gradient is not a Lipschitz operator, Lemma \ref{lemma:staerr:liprad} does not work and we can not bound the Rademacher complexity in the same way.
\begin{lemma}\label{lemma:staerr:Rad1}
	\begin{equation*}
		\mathbb{E}_{\boldsymbol{Z}_{n}}\left[ \sup_{u\in\mathcal{N}^{2}}\left| \mathcal{L}_{\lambda,1}(u)-\widehat{\mathcal{L}}_{\lambda,1}(u) \right| \right]\leq \mathfrak{R}(\mathcal{N}^{1,2}). \\
	\end{equation*}
\end{lemma}
\begin{proof}
	Based on the symmetrization method, we have
	\begin{equation}\label{eq:upr12}
		\mathbb{E}_{\boldsymbol{Z}_n}\left[\sup_{u \in \mathcal{N}^2} \left|\mathcal{L}_{\lambda,1}(u) - \widehat{\mathcal{L}}_{\lambda,1}(u)\right|\right] \leq \mathbb{E}_{\boldsymbol{Z}_n,\Sigma_n}\left[\sup_{u\in \mathcal{N}^2}\frac{1}{n}\left|\sum_{i} \sigma_i \|\nabla u(Z_i)\|^2\right|\right]
	\end{equation}
	The proof of \eqref{eq:upr12} is a direct consequence of the following claim.
	\par\noindent{\textbf{Claim}: Let $u$ be a function implemented by a $\mathrm{ReLU}^{2}$ network with depth $\mathcal{D}$ and width $\mathcal{W}$. Then $\|\nabla u\|_2^2$ can be implemented by a $\mathrm{ReLU}$-$\mathrm{ReLU}^{2}$ network with depth $\mathcal{D}+3$ and width $d\left(\mathcal{D}+2\right)\mathcal{W}$.}
	
	Denote $\mathrm{ReLU}$ and $\mathrm{ReLU}^{2}$ as $\sigma_1$ and $\sigma_2$, respectively.
	As long as we show that each partial derivative  $D_iu(i=1,2,\cdots,d)$ can be implemented by a $\mathrm{ReLU}$-$\mathrm{ReLU}^{2}$ network respectively, we can easily obtain the network we desire, since, $\|\nabla u\|_2^2=\sum_{i=1}^{d}\left|D_i u\right|^2$  and the square function can be implemented by  $x^2=\sigma_2(x)+\sigma_2(-x)$.
	
	Now we show that for any $i=1,2,\cdots,d$, $D_iu$ can be implemented by a $\mathrm{ReLU}$-$\mathrm{ReLU}^{2}$ network. We deal with the first two layers in details since there  are  a  little bit difference for the first two layer and apply induction for layers $k\geq3$. For the first layer, since $\sigma_2^{'}(x)=2\sigma_1(x)$, we have for any $q=1,2\cdots,n_1$
	\begin{equation*}
		D_iu_q^{(1)}=D_i\sigma_2\left(\sum_{j=1}^{d}a_{qj}^{(1)}x_j+b_q^{(1)}\right)
		=2\sigma_1\left(\sum_{j=1}^{d}a_{qj}^{(1)}x_j+b_q^{(1)}\right)\cdot a_{qi}^{(1)}
	\end{equation*}
	Hence $D_iu_q^{(1)}$ can be implemented by a $\mathrm{ReLU}$-$\mathrm{ReLU}^{2}$ network with depth $2$ and width $1$. For the second layer,
	\begin{equation*}
		D_iu_q^{(2)}=D_i\sigma_2\left(\sum_{j=1}^{n_1}a_{qj}^{(2)}u_j^{(1)}+b_q^{(2)}\right)
		=2\sigma_1\left(\sum_{j=1}^{n_1}a_{qj}^{(2)}u_j^{(1)}+b_q^{(2)}\right)\cdot\sum_{j=1}^{n_1}a_{qj}^{(2)}D_iu_j^{(1)}
	\end{equation*}
	Since $\sigma_1\left(\sum_{j=1}^{n_1}a_{qj}^{(2)}u_j^{(1)}+b_q^{(2)}\right)$ and $\sum_{j=1}^{n_1}a_{qj}^{(2)}D_iu_j^{(1)}$ can be implemented by two $\mathrm{ReLU}$-$\mathrm{ReLU}^{2}$ subnetworks, respectively, and the multiplication can also be implemented by
	\begin{equation*}
		\begin{split}
			x\cdot y&=\frac{1}{4}\left[(x+y)^2-(x-y)^2\right]\\
			&=\frac{1}{4}\left[\sigma_2(x+y)+\sigma_2(-x-y)-\sigma_2(x-y)-\sigma_2(-x+y)\right],
		\end{split}
	\end{equation*}
	we conclude that $D_iu_q^{(2)}$ can be implemented by a $\mathrm{ReLU}$-$\mathrm{ReLU}^{2}$ network. We have $$\mathcal{D}\left(\sigma_1\left(\sum_{j=1}^{n_1}a_{qj}^{(2)}u_j^{(1)}+b_q^{(2)}\right)\right)=3, \mathcal{W}\left(\sigma_1\left(\sum_{j=1}^{n_1}a_{qj}^{(2)}u_j^{(1)}+b_q^{(2)}\right)\right)\leq\mathcal{W}$$ and $$\mathcal{D}\left(\sum_{j=1}^{n_1}a_{qj}^{(2)}D_iu_j^{(1)}\right)=2, \mathcal{W}\left(\sum_{j=1}^{n_1}a_{qj}^{(2)}D_iu_j^{(1)}\right)\leq\mathcal{W}.$$ Thus $\mathcal{D}\left(D_iu_q^{(2)}\right)=4,$ $\mathcal{W}\left(D_iu_q^{(2)}\right)\leq\max\{2\mathcal{W},4\}$.
	
	Now we apply induction for layers $k\geq3$. For the third layer,
	\begin{equation*}
		D_iu_q^{(3)}=D_i\sigma_2\left(\sum_{j=1}^{n_2}a_{qj}^{(3)}u_j^{(2)}+b_q^{(3)}\right)
		=2\sigma_1\left(\sum_{j=1}^{n_2}a_{qj}^{(3)}u_j^{(2)}+b_q^{(3)}\right)\cdot\sum_{j=1}^{n_2}a_{qj}^{(3)}D_iu_j^{(2)}.
	\end{equation*}
	Since
	\begin{equation*}
		\mathcal{D}\left(\sigma_1\left(\sum_{j=1}^{n_2}a_{qj}^{(3)}u_j^{(2)}+b_q^{(3)}\right)\right)=4, \mathcal{W}\left(\sigma_1\left(\sum_{j=1}^{n_2}a_{qj}^{(3)}u_j^{(2)}+b_q^{(3)}\right)\right)\leq\mathcal{W}
	\end{equation*}
	and $$\mathcal{D}\left(\sum_{j=1}^{n_2}a_{qj}^{(3)}D_iu_j^{(2)}\right)=4, \mathcal{W}\left(\sum_{j=1}^{n_1}a_{qj}^{(3)}D_iu_j^{(2)}\right)\leq\max\{2\mathcal{W},4\mathcal{W}\}=4\mathcal{W},$$ we conclude that $D_iu_q^{(3)}$ can be implemented by a $\mathrm{ReLU}$-$\mathrm{ReLU}^{2}$ network and $\mathcal{D}\left(D_iu_q^{(3)}\right)=5$, $\mathcal{W}\left(D_iu_q^{(3)}\right)\leq\max\{5\mathcal{W},4\}=5\mathcal{W}$.
	
	We assume that $D_iu_q^{(k)}(q=1,2,\cdots,n_k)$ can be implemented by a $\mathrm{ReLU}$-$\mathrm{ReLU}^{2}$ network and $\mathcal{D}\left(D_iu_q^{(k)}\right)=k+2$, $\mathcal{W}\left(D_iu_q^{(3)}\right)\leq(k+2)\mathcal{W}$. For the $(k+1)-$th layer,
	\begin{align*}
		&D_iu_q^{(k+1)}=D_i\sigma_2\left(\sum_{j=1}^{n_k}a_{qj}^{(k+1)}u_j^{(k)}+b_q^{(k+1)}\right)\\
		&=2\sigma_1\left(\sum_{j=1}^{n_k}a_{qj}^{(k+1)}u_j^{(k)}+b_q^{(k+1)}\right)\cdot\sum_{j=1}^{n_k}a_{qj}^{(k+1)}D_iu_j^{(k)}.
	\end{align*}
	Since
	\begin{equation*}
		\begin{aligned}
			\mathcal{D}\left(\sigma_1\left(\sum_{j=1}^{n_k}a_{qj}^{(k+1)}u_j^{(k)}+b_q^{(k+1)}\right)\right)&=k+2, \\
			\mathcal{W}\left(\sigma_1\left(\sum_{j=1}^{n_k}a_{qj}^{(k+1)}u_j^{(k)}+b_q^{(k+1)}\right)\right)&\leq\mathcal{W},
		\end{aligned}
	\end{equation*}
	and
	\begin{equation*}
		\begin{aligned}
			\mathcal{D}\left(\sum_{j=1}^{n_k}a_{qj}^{(k+1)}D_iu_j^{(k)}\right)&=k+2, \\ \mathcal{W}\left(\sum_{j=1}^{n_k}a_{qj}^{(k+1)}D_iu_j^{(k)}\right)&\leq\max\{(k+2)\mathcal{W},4\mathcal{W}\}=(k+2)\mathcal{W},
		\end{aligned}
	\end{equation*}
	we conclude that $D_iu_q^{(k+1)}$ can be implemented by a $\mathrm{ReLU}$-$\mathrm{ReLU}^{2}$ network and $\mathcal{D}\left(D_iu_q^{(k+1)}\right)=k+3$, $\mathcal{W}\left(D_iu_q^{(k+1)}\right)\leq\max\{(k+3)\mathcal{W},4\}=(k+3)\mathcal{W}$.
	
	Hence we derive that $D_iu=D_iu_1^{\mathcal{D}}$ can be implemented by a $\mathrm{ReLU}$-$\mathrm{ReLU}^{2}$ network and $\mathcal{D}\left(D_iu\right)=\mathcal{D}+2$, $\mathcal{W}\left(D_iu\right)\leq\left(\mathcal{D}+2\right)\mathcal{W}$. Finally we obtain that $\mathcal{D}\left(\|\nabla u\|^2\right)=\mathcal{D}+3$, $\mathcal{W}\left(\|\nabla u\|^2\right)\leq d\left(\mathcal{D}+2\right)\mathcal{W}$.
\end{proof}

\par We are now in a position to bound the Rademacher complexity of $\mathcal{N}^{2}$ and $\mathcal{N}^{1,2}$. To obtain the estimation, we need to introduce covering number, VC-dimension, pseudo-dimension and recall several properties of them.

\begin{definition}
Suppose that $W \subset \mathbb{R}^{n} .$ For any $\epsilon>0$, let $V \subset \mathbb{R}^{n}$ be a $\epsilon$ -cover of $W$ with respect to the distance $d_{\infty}$, that is, for any $w \in W$, there exists a $v \in V$ such that $d_{\infty}(u, v)<\epsilon$, where $d_{\infty}$ is defined by
$$
d_{\infty}(u, v):=\|u-v\|_{\infty} .
$$
The covering number $\mathcal{C}\left(\epsilon, W, d_{\infty}\right)$ is defined to be the minimum cardinality among all $\epsilon$-cover of $W$ with respect to the distance $d_{\infty}$.	
\end{definition}

\begin{definition}
Suppose that $\mathcal{N}$ is a class of functions from $\Omega$ to $\mathbb{R} .$ Given $n$ sample $\mathbf{Z}_{n}=\left(Z_{1}, Z_{2}, \cdots, Z_{n}\right) \in \Omega^{n},\left.\mathcal{N}\right|_{\mathbf{z}_{n}} \subset \mathbb{R}^{n}$ is defined by
\begin{equation*}
\mathcal{N} \mid \mathbf{z}_{n}=\left\{\left(u\left(Z_{1}\right), u\left(Z_{2}\right), \cdots, u\left(Z_{n}\right)\right): u \in \mathcal{N}\right\}.
\end{equation*}
The uniform covering number $\mathcal{C}_{\infty}(\epsilon, \mathcal{N}, n)$ is defined by
\begin{equation*}
\mathcal{C}_{\infty}(\epsilon, \mathcal{N}, n)=\max _{\mathbf{Z}_{n} \in \Omega^{n}} \mathcal{C}\left(\epsilon, \mathcal{N} \mid \mathbf{z}_{n}, d_{\infty}\right)
\end{equation*}
\end{definition}

\par Next we give a upper bound of $\mathfrak{R}\left(\mathcal{N}\right)$ in terms of the covering number of $\mathcal{N}$ by using the Dudley's entropy formula \cite{dudley}.

\begin{lemma}[Massart's finite class lemma \cite{boucheron2013concentration}]\label{Mfinite}
	For any finite set $V\in\mathbb{R}^{n}$ with diameter $D=\sum_{v\in V}\|v\|_{2}$, then
	\begin{equation*}
		\mathbb{E}_{\Sigma_{n}}\left[ \sup_{v\in V}\frac{1}{n}\left| \sum_{i}\sigma_{i}v_{i} \right| \right] \leq \frac{D}{n}\sqrt{2\log(2|V|)}.
	\end{equation*}
\end{lemma}
We give an upper bound of $\mathfrak{R}(\mathcal{N})$ in terms of the covering number by using the Dudley's entropy formula \cite{dudley}.
\begin{lemma}[Dudley's entropy formula \cite{dudley}]\label{lemma:staerr:Rad_cover}
	Assume $0\in \mathcal{N}$ and the diameter of $\mathcal{N}$ is less than $\mathcal{B}$, i.e., $\|u\|_{L^\infty(\Omega)} \leq \mathcal{B}, \forall u \in \mathcal{N}$.  Then
	$$\mathfrak{R}(\mathcal{N}) \leq  \inf_{0<\delta<\mathcal{B}}\left(4 \delta+\frac{12}{\sqrt{n}} \int_{\delta}^{\mathcal{B}} \sqrt{\log (2\mathcal{C}\left( \varepsilon, \mathcal{N}, n\right))}\mathrm{d}\varepsilon\right).$$
\end{lemma}
\begin{proof}
	By definition
	\begin{equation*}
		\mathfrak{R}(\mathcal{N}) = \mathfrak{R}(\mathcal{N}|_{\boldsymbol{Z}_n}) = \mathbb{E}_{\mathbb{Z}_n} \left[\mathbb{E}_{\Sigma}\left.\left[\sup_{u\in \mathcal{N}}\frac{1}{n}\left|\sum_{i} \sigma_i u(Z_i)\right| \ \right|{\boldsymbol{Z}_n} \right] \right].
	\end{equation*}
	Thus, it suffice to show $$\mathbb{E}_{\Sigma}\left[\sup_{u\in \mathcal{N}}\frac{1}{n}\left|\sum_{i} \sigma_i u(Z_i)\right|\right]
	\leq  \inf_{0<\delta< \mathcal{B}}\left(4 \delta+\frac{12}{\sqrt{n}} \int_{\delta}^{\mathcal{B}} \sqrt{\log \mathcal{C}\left( \varepsilon, \mathcal{N}^2, n\right)}\mathrm{d}\varepsilon\right)$$ by conditioning  on $\boldsymbol{Z}_n$.
	Given an positive integer $K$, let $\varepsilon_k = 2^{-k+1}\mathcal{B}$, $k=1,...K$. Let $C_k$ be a cover of $\mathcal{N}|_{\boldsymbol{Z}_n}\subseteq\mathbb{R}^n$ whose   covering number is denoted as  $\mathcal{C}(\varepsilon_k,\mathcal{N}|_{\boldsymbol{Z}_n},d_{\infty})$. Then, by definition,
	$\forall u \in \mathcal{N},$ there $\exists$ $c^k \in C_k$ such that $$d_{\infty}(u|_{\boldsymbol{Z}_n},c^k) =\max\{|u(Z_i)-c^k_i|, i =1,...,n\}\leq \varepsilon_k, k =1,...,K.$$ Moreover, we denote
	the best approximate element of $u$ in $C_k$ with respect to $d_{\infty}$ as $c^k(u)$.
	Then, \begin{align*}
		& \mathbb{E}_{\Sigma}\left[\sup_{u\in \mathcal{N}}\frac{1}{n}\left|\sum_{i=1}^n \sigma_i u(Z_i)\right|\right]  \\
		& = \mathbb{E}_{\Sigma}\left[\sup_{u\in \mathcal{N}}\frac{1}{n}\left|\sum_{i=1}^n \sigma_i (u(Z_i)-c^K_i(u)) +\sum_{j=1}^{K-1}\sum_{i=1}^n \sigma_i(c^j_i(u)-c^{j+1}_i(u))+\sum_{i=1}^n \sigma_i c^{1}_{i}(u)\right|\right] \\
		& \leq  \mathbb{E}_{\Sigma}\left[\sup_{u\in \mathcal{N}}\frac{1}{n}\left|\sum_{i=1}^n \sigma_i (u(Z_i)-c^K_i(u))\right|\right]
		+\sum_{j=1}^{K-1} \mathbb{E}_{\Sigma}\left[\sup_{u\in \mathcal{N}}\frac{1}{n}\left|\sum_{i=1}^n \sigma_i(c^j_i(u)-c^{j+1}_i(u))\right|\right] \\
		& + \mathbb{E}_{\Sigma}\left[\sup_{u\in \mathcal{N}}\frac{1}{n}\left|\sum_{i=1}^n \sigma_i c^{1}_{i}(u)\right|\right].
	\end{align*}
	Since $0\in \mathcal{N}$, and the diameter of $\mathcal{N}$ is smaller than $\mathcal{B}$, we can choose $C_1 = \{0\}$ such that
	the third term in the above display vanishes.
	By  H\"{o}lder's inequality, we deduce that the first term can be bounded by $\varepsilon_K$ as follows.
	\begin{align*}
		& \mathbb{E}_{\Sigma}\left[\sup_{u\in \mathcal{N}}\frac{1}{n}\left|\sum_{i=1}^n \sigma_i (u(Z_i)-c^K_i(u))\right|\right]   \\
		& \leq \mathbb{E}_{\Sigma}\left[\sup_{u\in \mathcal{N}}\frac{1}{n}\left(\sum_{i=1}^n|\sigma_i|\right)\left(\sum_{i=1}^n\max_{i=1,...,n}\left\{\left|u(Z_i)-c^K_i(u)\right|\right\}\right)\right] \\
		& \leq \varepsilon_K.
	\end{align*}
	Let $V_j = \{c^j(u)-c^{j+1}(u): u\in \mathcal{N}\}$. Then by definition, the number of elements in $V_j$ and $C_j$ satisfying
	$$|V_j|\leq |C_j||C_{j+1}|\leq |C_{j+1}|^2.$$ And the  diameter of $V_j$ denoted as $D_j$ can be bounded as
	\begin{align*}
		& D_j = \sup_{v\in V_j}\|v\|_2  \leq \sqrt{n}\sup_{u\in \mathcal{N}} \|c^j(u)-c^{j+1}(u)\|_{\infty} \\
		& \leq \sqrt{n}\sup_{u\in \mathcal{N}} \|c^j(u)-u\|_{\infty} + \|u-c^{j+1}(u)\|_{\infty}            \\
		& \leq \sqrt{n}(\varepsilon_j + \varepsilon_{j+1})                                                  \\
		& \leq 3\sqrt{n} \varepsilon_{j+1}.
	\end{align*}
	Then,
	\begin{align*}
		& \mathbb{E}_{\Sigma}[\sup_{u\in \mathcal{N}}\frac{1}{n}|\sum_{j=1}^{K-1}\sum_{i=1}^n \sigma_i(c^j_i(u)-c^{j+1}_i(u))|]\leq \sum_{j=1}^{K-1} \mathbb{E}_{\Sigma}[\sup_{v\in V_j}\frac{1}{n}|\sum_{i=1}^n \sigma_iv_j|] \\
		& \leq \sum_{j=1}^{K-1} \frac{D_j}{n}\sqrt{2\log (2|V_j|)}                                                                                                                                                             \\
		& \leq \sum_{j=1}^{K-1}  \frac{6\varepsilon_{j+1}}{\sqrt{n}}\sqrt{\log (2|C_{j+1}|)},
	\end{align*}
	where we use triangle inequality in the  first inequality, and use Lemma \ref{Mfinite} in the second
	inequality.
	Putting all the above estimates together, we get
	\begin{align*}
		& \mathbb{E}_{\Sigma}\left[\sup_{u\in \mathcal{N}}\frac{1}{n}\left|\sum_{i} \sigma_i u(Z_i)\right|\right]\leq  \varepsilon_K + \sum_{j=1}^{K-1}  \frac{6\varepsilon_{j+1}}{\sqrt{n}}\sqrt{\log (2|C_{j+1}|)} \\
		& \leq \varepsilon_K + \sum_{j=1}^{K}  \frac{12(\varepsilon_j-\varepsilon_{j+1})}{\sqrt{n}}\sqrt{\log (2 \mathcal{C}\left( \varepsilon_j, \mathcal{N}, n\right))}                      \\
		& \leq  \varepsilon_K +   \frac{12}{\sqrt{n}} \int_{\varepsilon_{K+1}}^{\mathcal{B}}\sqrt{\log (2 \mathcal{C}\left( \varepsilon, \mathcal{N}, n\right))}\mathrm{d}\varepsilon          \\
		& \leq \inf_{0<\delta<\mathcal{B}}(4 \delta+\frac{12}{\sqrt{n}} \int_{\delta}^{\mathcal{B}} \sqrt{\log (2\mathcal{C}\left( \varepsilon, \mathcal{N}, n\right))}\mathrm{d}\varepsilon).
	\end{align*}
	where, last inequality holds since for $0<\delta< \mathcal{B}$, we can choose $K$ to be the largest integer such that  $\varepsilon_{K+1} >\delta$, at this time $\varepsilon_K\leq  4 \varepsilon_{K+2} \leq 4 \delta.$
\end{proof}

\begin{definition}
Let $\mathcal{N}$ be a set of functions from $X=\Omega(\partial \Omega)$ to $\{0,1\} .$ Suppose that $S=\left\{x_{1}, x_{2}, \cdots, x_{n}\right\} \subset X .$ We say that $S$ is shattered by $\mathcal{N}$ if for any $b \in\{0,1\}^{n}$, there exists a $u \in \mathcal{N}$ satisfying
\begin{equation*}
u\left(x_{i}\right)=b_{i}, \quad i=1,2, \ldots, n.
\end{equation*}
\end{definition}

\begin{definition}
The VC-dimension of $\mathcal{N}$, denoted as $\operatorname{VCdim}(\mathcal{N})$, is defined to be the maximum cardinality among all sets shattered by $\mathcal{N}$.
\end{definition}

VC-dimension reflects the capability of a class of functions to perform binary classification of points. The larger VC-dimension is, the stronger the capability to perform binary classification is. For more diseussion of VC-dimension, readers are referred to \cite{anthony2009neural}.

For real-valued functions, we can generalize the concept of VC-dimension into pseudo-dimension \cite{anthony2009neural}.

\begin{definition}
Let $\mathcal{N}$ be a set of functions from $X$ to $\mathbb{R} .$ Suppose that $S=\left\{x_{1}, x_{2}, \cdots, x_{n}\right\} \subset$
$X .$ We say that $S$ is pseudo-shattered by $\mathcal{N}$ if there exists $y_{1}, y_{2}, \cdots, y_{n}$ such that for any $b \in\{0,1\}^{n}$, there exists a $u \in \mathcal{N}$ satisfying
$$
\operatorname{sign}\left(u\left(x_{i}\right)-y_{i}\right)=b_{i}, \quad i=1,2, \ldots, n
$$
and we say that $\left\{y_{i}\right\}_{i=1}^{n}$ witnesses the shattering.
\end{definition}

\begin{definition}
The pseudo-dimension of $\mathcal{N}$, denoted as $\operatorname{Pdim}(\mathcal{N})$, is defined to be the maximum cardinality among all sets pseudo-shattered by $\mathcal{N}$.
\end{definition}

\par The following proposition showing a relationship between uniform covering number and pseudo-dimension.
\begin{lemma}\label{lemma:staerr:cover_Pdim}
	Let $\mathcal{N}$ be a set of real functions from a domain $X$ to the bounded interval $[0, \mathcal{B}]$. Let $\varepsilon>0$. Then
	\begin{equation*}
		\mathcal{C}(\varepsilon, \mathcal{N}, n) \leq \sum_{i=1}^{\mathrm{Pdim}(\mathcal{N})}
		\begin{pmatrix}
			n \\
			i
		\end{pmatrix}
		\left(\frac{\mathcal{B}}{\varepsilon}\right)^{i}
	\end{equation*}
	which is less than $\left(\frac{en\mathcal{B}}{\varepsilon\cdot\mathrm{Pdim}(\mathcal{N})}\right)^{\mathrm{Pdim}(\mathcal{N})}$ for $n\geq\mathrm{Pdim}(\mathcal{N})$.
\end{lemma}
\begin{proof}
	See Theorem 12.2 in \cite{anthony2009neural}.
\end{proof}

\par We now present the bound of pseudo-dimension for the $\mathcal{N}^{2}$ and $\mathcal{N}^{1,2}$.
\begin{lemma}\label{lemma:covering_number}
	Let $p_1,\cdots,p_m$ be polynomials with $n$ variables  of degree at most $d$. If $n\leq m$, then
	\begin{equation*}
		|\{(\operatorname{sign}(p_1(x)),\cdots,\operatorname{sign}(p_m(x))):x\in\mathbb{R}^n\}| \leq 2\left(\frac{2emd}{n}\right)^n
	\end{equation*}
\end{lemma}
\begin{proof}
	See Theorem 8.3 in \cite{anthony2009neural}.
\end{proof}

\begin{lemma}\label{lemma:staerr:pdim}
	Let $\mathcal{N}$ be a set of functions that
	\begin{itemize}
		\item[(i)] can be implemented by a neural network with depth no more than $\mathcal{D}$ and width no more than $\mathcal{W}$, and
		\item[(ii)] the activation function in each unit be the $\mathrm{ReLU}$  or the $\mathrm{ReLU}^{2}$.
	\end{itemize}
	Then
	\begin{equation*}
		\operatorname{Pdim}(\mathcal{N})=\mathcal{O}(\mathcal{D}^2\mathcal{W}^2(\mathcal{D}+\log\mathcal{W})).
	\end{equation*}
\end{lemma}

\begin{proof}
	The argument is follows from the  proof of Theorem 6 in \cite{bartlett2019nearly}.  The result stated here is somewhat stronger then Theorem 6 in \cite{bartlett2019nearly} since $\mathrm{VCdim}(\operatorname{sign}(\mathcal{N}))\leq \mathrm{Pdim}(\mathcal{N})$.
	
	We consider a new set of functions:
	\begin{equation*}
		\mathcal{\widetilde{N}}=\{\widetilde{u}(x,y)=\operatorname{sign}(u(x)-y): u\in\mathcal{H}\}
	\end{equation*}
	It is clear that $\mathrm{Pdim}(\mathcal{N})\leq\mathrm{VCdim}(\mathcal{\widetilde{N}})$. We now bound the VC-dimension of $\mathcal{\widetilde{N}}$. Denoting $\mathcal{M}$ as the total number of parameters(weights and biases) in the neural network implementing functions in $\mathcal{N}$, in our case we want to derive the uniform bound for
	\begin{equation*}
		K_{\{x_i\},\{y_i\}}(m):=|\{(\operatorname{sign}(f(x_{1}, a)-y_1), \ldots, \operatorname{sign}(u(x_{m}, a)-y_m)): a \in \mathbb{R}^{\mathcal{M}}\}|
	\end{equation*}
	over all $\{x_i\}_{i=1}^{m}\subset X$ and $\{y_i\}_{i=1}^{m}\subset\mathbb{R}$. Actually the maximum of $K_{\{x_i\},\{y_i\}}(m)$ over all $\{x_i\}_{i=1}^{m}\subset X$ and $\{y_i\}_{i=1}^{m}\subset\mathbb{R}$ is the growth function $\mathcal{G}_{\mathcal{\widetilde{N}}}(m)$.
	In order to apply Lemma \ref{lemma:covering_number}, we partition the parameter space $\mathbb{R}^{\mathcal{M}}$ into several subsets to ensure that in each subset $u(x_i,a)-y_i$ is a polynomial with respcet to $a$ without any breakpoints. In fact, our partition is exactly the same as the partition in \cite{bartlett2019nearly}. Denote the partition as $\{P_1,P_2,\cdots,P_N\}$ with some integer $N$ satisfying
	\begin{equation} \label{pdimb1}
		N\leq\prod_{i=1}^{\mathcal{D}-1}2\left(\frac{2emk_i(1+(i-1)2^{i-1})}{\mathcal{M}_i}\right)^{\mathcal{M}_i}
	\end{equation}
	where $k_i$ and $\mathcal{M}_i$ denotes the number of units at the $i$th layer and the total number of parameters at the inputs to units in all the layers up to layer $i$ of the neural network implementing functions in $\mathcal{N}$, respectively. See \cite{bartlett2019nearly} for the construction of the partition. Obviously we have
	\begin{equation} \label{pdimb2}
		K_{\{x_i\},\{y_i\}}(m)\leq\sum_{i=1}^{N}|\{(\operatorname{sign}(u(x_{1}, a)-y_1), \cdots, \operatorname{sign}(u(x_{m}, a)-y_m)): a\in P_i\}|
	\end{equation}
	Note that $u(x_i,a)-y_i$ is a polynomial with respect to $a$ with degree the same as the degree of $u(x_i,a)$, which is equal to $1 + (\mathcal{D}-1)2^{\mathcal{D}-1}$ as shown in \cite{bartlett2019nearly}.  Hence by Lemma $\ref{lemma:covering_number}$, we have
	\begin{align}
		& |\{(\operatorname{sign}(u(x_{1}, a)-y_1), \cdots, \operatorname{sign}(u(x_{m}, a)-y_m)): a\in P_i\}|\nonumber                          \\
		& \leq2\left(\frac{2em(1+(\mathcal{D}-1)2^{\mathcal{D}-1})}{\mathcal{M}_{\mathcal{D}}}\right)^{\mathcal{M}_{\mathcal{D}}}\label{pdimb3}.
	\end{align}
	Combining $(\ref{pdimb1}), (\ref{pdimb2}), (\ref{pdimb3})$ yields
	\begin{equation*}
		K_{\{x_i\},\{y_i\}}(m)\leq\prod_{i=1}^{\mathcal{D}}2\left(\frac{2emk_i(1+(i-1)2^{i-1})}{\mathcal{M}_i}\right)^{\mathcal{M}_i}.
	\end{equation*}
	We then have
	\begin{equation*}
		\mathcal{G}_{\mathcal{\widetilde{N}}}(m)\leq\prod_{i=1}^{\mathcal{D}}2\left(\frac{2emk_i(1+(i-1)2^{i-1})}{\mathcal{M}_i}\right)^{\mathcal{M}_i},
	\end{equation*}
	since the maximum of $K_{\{x_i\},\{y_i\}}(m)$ over all $\{x_i\}_{i=1}^{m}\subset X$ and $\{y_i\}_{i=1}^{m}\subset\mathbb{R}$ is the growth function $\mathcal{G}_{\mathcal{\widetilde{N}}}(m)$. Some algebras  as that of  the proof of Theorem 6 in \cite{bartlett2019nearly}, we obtain
	\begin{equation*}
		\mathrm{Pdim}(\mathcal{N})\leq\mathcal{O}\left(\mathcal{D}^2\mathcal{W}^2 \log \mathcal{U}+\mathcal{D}^3\mathcal{W}^2\right)
		=\mathcal{O}\left(\mathcal{D}^2\mathcal{W}^2\left( \mathcal{D}+\log \mathcal{W}\right)\right)
	\end{equation*}
	where $\mathcal{U}$ refers to the number of units of the neural network implementing functions in $\mathcal{N}$.
\end{proof}

\par With the help of above preparations, the statistical error can easily be bounded by a tedious calculation.

\begin{theorem}\label{thm:staerr}
	Let $\mathcal{D}$ and $\mathcal{W}$ be the depth and width of the network respectively, then
	\begin{equation*}
		\begin{aligned}
			\mathcal{E}_{sta} \leq & C_{c_3}d(\mathcal{D}+3)(\mathcal{D}+2)\mathcal{W}\sqrt{\mathcal{D}+3+\log (d(\mathcal{D}+2)\mathcal{W})}\left(\frac{\log n}{n}\right)^{1/2} \\
			& +C_{c_3}d\mathcal{D}\mathcal{W}\sqrt{\mathcal{D}+\log\mathcal{W}}\left(\frac{\log n}{n}\right)^{1/2}\lambda.
		\end{aligned}
	\end{equation*}
	where $n$ is the number of training samples on  both the domain and the boundary.
\end{theorem}
\begin{proof}
	In order to apply Lemma \ref{lemma:staerr:Rad_cover}, we need to handle the term
	\begin{equation*}
		\begin{split}
			&\frac{1}{\sqrt{n}}\int_{\delta}^{\mathcal{B}}\sqrt{\log(2\mathcal{C}(\epsilon,\mathcal{N},n))}d\epsilon \\
			&\leq \frac{\mathcal{B}}{\sqrt{n}}+ \frac{1}{\sqrt{n}}\int_{\delta}^{B}\sqrt{\log\left(\frac{en\mathcal{B}}{\epsilon\cdot\mathrm{Pdim}(\mathcal{N})}\right)^{\mathrm{Pdim}(\mathcal{N})}}d\epsilon\\
			&\leq \frac{\mathcal{B}}{\sqrt{n}}+ \left(\frac{\mathrm{Pdim}(\mathcal{N})}{n}\right)^{1/2}\int_{\delta}^{B}\sqrt{\log\left(\frac{en\mathcal{B}}{\epsilon\cdot\mathrm{Pdim}(\mathcal{N})}\right)}d\epsilon
		\end{split}
	\end{equation*}
	where in the first inequality  we use Lemma $\ref{lemma:staerr:cover_Pdim}$. Now we calculate the integral. Set
	\begin{equation*}
		t=\sqrt{\log\left(\frac{en\mathcal{B}}{\epsilon\cdot\mathrm{Pdim}(\mathcal{N})}\right)}
	\end{equation*}
	then $\epsilon=\frac{en\mathcal{B}}{\mathrm{Pdim}(\mathcal{N})}\cdot e^{-t^2}$. Denote $t_1=\sqrt{\log\left(\frac{en\mathcal{B}}{\mathcal{B}\cdot\mathrm{Pdim}(\mathcal{N})}\right)}$, $t_2=\sqrt{\log\left(\frac{en\mathcal{B}}{\delta\cdot\mathrm{Pdim}(\mathcal{N})}\right)}$. And
	\begin{equation*}
		\begin{split}
			&\int_{\delta}^{\mathcal{B}}\sqrt{\log\left(\frac{en\mathcal{B}}{\epsilon\cdot\mathrm{Pdim}(\mathcal{N})}\right)}d\epsilon
			=\frac{2en\mathcal{B}}{\mathrm{Pdim}(\mathcal{N})}\int_{t_1}^{t_2}t^2e^{-t^2}dt\\
			&=\frac{2en\mathcal{B}}{\mathrm{Pdim}(\mathcal{N})}\int_{t_1}^{t_2}t\left(\frac{-e^{-t^2}}{2}\right)'dt\\
			&=\frac{en\mathcal{B}}{\mathrm{Pdim}(\mathcal{N})}\left[t_1e^{-t_1^2}-t_2e^{-t_2^2}+\int_{t_1}^{t_2}e^{-t^2}dt\right]\\
			&\leq\frac{en\mathcal{B}}{\mathrm{Pdim}(\mathcal{N})}\left[t_1e^{-t_1^2}-t_2e^{-t_2^2}+(t_2-t_1)e^{-t_1^2}\right]\\
			&\leq\frac{en\mathcal{B}}{\mathrm{Pdim}(\mathcal{N})}\cdot t_2e^{-t_1^2}=\mathcal{B}\sqrt{\log\left(\frac{en\mathcal{B}}{\delta\cdot\mathrm{Pdim}(\mathcal{N})}\right)}
		\end{split}
	\end{equation*}
	Choosing $\delta=\mathcal{B}\left(\frac{\mathrm{Pdim}(\mathcal{N})}{n}\right)^{1/2}\leq\mathcal{B}$, by Lemma \ref{lemma:staerr:Rad_cover} and the above display, we get for both  $\mathcal{N} = \mathcal{N}^2$ and $\mathcal{N} = \mathcal{N}^{1,2}$  there holds
	\begin{align}
		& \mathfrak{R}(\mathcal{N}) \leq 4\delta+\frac{12}{\sqrt{n}}\int_{\delta}^{\mathcal{B}}\sqrt{\log(2\mathcal{C}(\epsilon,\mathcal{N},n))}d\epsilon \nonumber                                                       \\
		& \leq4\delta+  \frac{12\mathcal{B}}{\sqrt{n}}+ 12\mathcal{B}\left(\frac{\mathrm{Pdim}(\mathcal{N})}{n}\right)^{1/2}\sqrt{\log\left(\frac{en\mathcal{B}}{\delta\cdot\mathrm{Pdim}(\mathcal{N})}\right)} \nonumber \\
		& \leq28\sqrt{\frac{3}{2}}\mathcal{B}\left(\frac{\mathrm{Pdim}(\mathcal{N})}{n}\right)^{1/2}\sqrt{\log\left(\frac{en}{\mathrm{Pdim}(\mathcal{N})}\right)} \label{emperr}.
	\end{align}
	Then by Lemma \ref{lemma:staerr:dec}, \ref{lemma:staerr:Rad_cover}, \ref{lemma:staerr:Rad234}, \ref{lemma:staerr:Rad1} and equation (\ref{emperr}), we have
	\begin{align*}
		& {\mathcal{E}_{sta}}  = 2\sup_{u \in \mathcal{N}^2} |\mathcal{L}(u) - \widehat{\mathcal{L}}(u) |                                                                             \\
		& \leq 2\mathfrak{R}(\mathcal{N}^{1,2})+2(2c_{3}^{2}+2c_{3})\mathfrak{R}(\mathcal{N}^{2})+2c_{3}^{2}\mathfrak{R}(\mathcal{N}^{2})\lambda                                      \\
		& \leq
		56\sqrt{\frac{3}{2}}\mathcal{B}\left(\frac{\mathrm{Pdim}(\mathcal{N}^{1,2})}{n}\right)^{1/2}\sqrt{\log\left(\frac{en}{\mathrm{Pdim}(\mathcal{N}^{1,2})}\right)}                \\
		& +56\sqrt{\frac{3}{2}}(2c_{3}^{2}+2c_{3})\mathcal{B}\left(\frac{\mathrm{Pdim}(\mathcal{N}^2)}{n}\right)^{1/2}\sqrt{\log\left(\frac{en}{\mathrm{Pdim}(\mathcal{N}^2)}\right)} \\
		& +56\sqrt{\frac{3}{2}}c_{3}^{2}\mathcal{B}\left(\frac{\mathrm{Pdim}(\mathcal{N}^2)}{n}\right)^{1/2}\sqrt{\log\left(\frac{en}{\mathrm{Pdim}(\mathcal{N}^2)}\right)}\lambda.
	\end{align*}
	Plugging the upper bound of $\mathrm{Pdim}$ derived in Lemma \ref{lemma:staerr:cover_Pdim} into the above display  and using the relationship of depth and width between $\mathcal{N}^2$ and $\mathcal{N}^{1,2}$, we get
	\begin{equation}\label{sterrf}
		\begin{aligned}
			\mathcal{E}_{sta} \leq & C_{c_3}d(\mathcal{D}+3)(\mathcal{D}+2)\mathcal{W}\sqrt{\mathcal{D}+3+\log (d(\mathcal{D}+2)\mathcal{W})}\left(\frac{\log n}{n}\right)^{1/2} \\
			&+C_{c_3}d\mathcal{D}\mathcal{W}\sqrt{\mathcal{D}+\log\mathcal{W}}\left(\frac{\log n}{n}\right)^{1/2}\lambda.
		\end{aligned}
	\end{equation}
\end{proof}

\subsection{Error from the boundary penalty method}
\par Although the Lemma \ref{lemma:lambda_convergence} shows the convergence property of Robin problem \eqref{eq:robin} as $\lambda\rightarrow\infty$
\begin{equation*}
	u_{\lambda}^{*} \rightarrow u^{*},
\end{equation*}
it says nothing about the convergence rate. In this section, we consider the error from the boundary penalty method. Roughly speaking, we bound the distance between the minimizer $u^{*}$ and $u_{\lambda}^{*}$ with respect to the penalty parameter $\lambda$.
\begin{theorem}\label{thm:penaltyerr}
	Suppose $u_{\lambda}^{*}$ is the minimizer of \eqref{eq:energy:penalty} and $u^{*}$ is the minimizer of \eqref{eq:energy}. Then
	\begin{equation*}
		\|u_{\lambda}^{*}-u^{*}\|_{H^{1}(\Omega)}\leq C_{c_{3},d}\lambda^{-1}.
	\end{equation*}
\end{theorem}
\begin{proof}
	\par Following the idea which is proposed in \cite{Maury2009Numerical} (proof of Proposition 2.3), we proceed to prove this theorem.
	For $v\in H^{1}(\Omega)$, we introduce
	\begin{equation}\label{eq:R_lambda}
		R_{\lambda}(v)=\frac{1}{2}a(u^{*}-v,u^{*}-v)+\frac{\lambda}{2}\int_{\partial\Omega}\left( -\frac{1}{\lambda}\frac{\partial u^{*}}{\partial n}-v \right)^{2}ds.
	\end{equation}
	Given $\varphi\in H^{1}(\Omega)$ such that $T\varphi=-\frac{\partial u^{*}}{\partial n}$, we set $w=\frac{1}{\lambda}\varphi+u^{*}$. Due to $u^{*}\in H^{1}_{0}(\Omega)$, it follows that
	\begin{equation}\label{eq:R_bound}
		R_{\lambda}(w) = \frac{1}{2\lambda^{2}}a\left( \varphi,\varphi \right)+\frac{\lambda}{2}\int_{\partial\Omega}(u^{*})^{2}ds = \frac{1}{2\lambda^{2}}a\left( \varphi,\varphi \right) \leq C\lambda^{-2},
	\end{equation}
	where $C$ is dependent only on $\mathcal{B}$, $w$ and $\Omega$.
	Apparently, \eqref{eq:R_lambda} can be written
	\begin{equation*}
		\begin{aligned}
			R_{\lambda}(v)  =&\frac{1}{2}a(u^{*},u^{*})-a(u^{*},v)+\frac{1}{2}a(v,v)+\frac{1}{2\lambda}\int_{\partial\Omega}\left(\frac{\partial u^{*}}{\partial n}\right)^{2}ds+\int_{\partial\Omega}\frac{\partial u^{*}}{\partial n}vds \\
			&+\frac{\lambda}{2}\int_{\partial\Omega}v^{2}ds \\
			=&\frac{1}{2}a(u^{*},u^{*})+\frac{1}{2}a(v,v)+\frac{1}{2\lambda}\int_{\partial\Omega}\left(\frac{\partial u^{*}}{\partial n}\right)^{2}ds+\frac{\lambda}{2}\int_{\partial\Omega}v^{2}ds-\int_{\Omega}fvdx \\
			=&\frac{1}{2}a(u^{*},u^{*})+\frac{1}{2\lambda}\int_{\partial\Omega}\left(\frac{\partial u^{*}}{\partial n}\right)^{2}ds+\mathcal{L}_{\lambda}(v),
		\end{aligned}
	\end{equation*}
	where the second equality comes from that
	\begin{equation}\label{eq:variational}
		a(u^{*},v)-\int_{\partial\Omega} \frac{\partial u^{*}}{\partial n}v ds=\int_{\Omega}fvdx , \ \forall v\in H^{1}(\Omega).
	\end{equation}
	Since $R_{\lambda}(v)=\mathcal{L}_{\lambda}(v)+const$, $u_{\lambda}^{*}$ is also the minimizer of $R_{\lambda}$ over $H^{1}(\Omega)$. Recall \eqref{eq:R_bound}, we obtain the estimation of $R_{\lambda}(u_{\lambda}^{*})$
	\begin{equation*}
		0 \leq R_{\lambda}(u_{\lambda}^{*}) = \frac{1}{2}a(u^{*}-u_{\lambda}^{*},u^{*}-u_{\lambda}^{*})+\frac{\lambda}{2}\int_{\partial\Omega}\left( -\frac{1}{\lambda}\frac{\partial u^{*}}{\partial n}-u_{\lambda}^{*} \right)^{2}ds \leq R_{\lambda}(w) \leq C\lambda^{-2}.
	\end{equation*}
	Now that $a(\cdot,\cdot)$ is coercive, we arrive at
	\begin{equation*}
		\|u^{*}-u_{\lambda}^{*}\|_{H^{1}(\Omega)}\leq C_{c_{3},d}\lambda^{-1}.
	\end{equation*}
\end{proof}
In \cite{hong2021rademacher}, they proved that the error 	$\|u_{\lambda}^{*}-u^{*}\|_{H^{1}(\Omega)}\leq \mathcal{O}(\lambda^{-1/2}),$ which is suboptimal comparing with the above results derived here.
In \cite{Mller2021ErrorEF,muller2020deep}, they proved the $\mathcal{O}(\lambda^{-1})$ bound under some unverifiable conditions.
\subsection{Convergence rate}

\par Note that for $\lambda\rightarrow \infty$ the approximation error $\mathcal{E}_{app}$ and the statistical error $\mathcal{E}_{sta}$ approach $\infty$ and for $\lambda\rightarrow 0$ the error from penalty blows up. Hence, there must be a trade off for choosing proper $\lambda$.
\begin{theorem}\label{thm:convergence}
	Let $u^{*}$ be the weak solution of \eqref{eq:dirichlet} with bounded $f\in L^{2}(\Omega)$, $w\in L^{\infty}(\Omega)$. $\widehat{u}_{\phi}$ is the minimizer of the discrete version of the associated Robin energy with parameter $\lambda$. Given $n$ be the number of training samples on the domain and the boundary, there is a $\mathrm{ReLU}^2$ network with depth and width as
	\begin{equation*}
		\mathcal{D} \leq \lceil\log_2d\rceil+3, \quad \mathcal{W}\leq \mathcal{O}\left(4d\left \lceil \left(\frac{n}{\log n}\right)^{\frac{1}{2(d+2)}}-4\right \rceil^d \right),
	\end{equation*}
	such that
	\begin{equation*}
		\begin{aligned}
			\mathbb{E}_{\boldsymbol{X},\boldsymbol{Y}}\left[\|\widehat{u}_{\phi}-u^{*}\|_{H^{1}(\Omega)}^2\right]
			\leq & C_{c_1,c_2,c_3,d} {\mathcal{O}} \left(n^{-\frac{1}{d+2}}(\log n)^{\frac{d+3}{d+2}}\right)  \\
			& +C_{c_1,c_2,c_3,d} {\mathcal{O}} \left(n^{-\frac{1}{d+2}}(\log n)^{\frac{d+3}{d+2}}\right)\lambda \\
			& + C_{c_{3},d}\lambda^{-2}.
		\end{aligned}
	\end{equation*}
	Furthermore, for
	\begin{equation*}
		\lambda \sim n^{\frac{1}{3(d+2)}}(\log n)^{-\frac{d+3}{3(d+2)}},
	\end{equation*}
	it holds that
	\begin{equation*}
		\mathbb{E}_{\boldsymbol{X},\boldsymbol{Y}}\left[\|\widehat{u}_{\phi}-u^{*}\|_{H^{1}(\Omega)}^2\right]
		\leq C_{c_1,c_2,c_3,d} {\mathcal{O}} \left(n^{-\frac{2}{3(d+2)}}\log n\right).
	\end{equation*}
\end{theorem}
\begin{proof}
	Combining Theorem \ref{thm:errdec}, Theorem \ref{thm:apperr} and Theorem \ref{thm:staerr}, we obtain by taking $\varepsilon^{2} = C_{c_1,c_2,c_3}\left(\frac{\log n}{n}\right)^{\frac{1}{d+2}}$
	\begin{equation*}
		\begin{aligned}
			&\mathbb{E}_{\boldsymbol{X},\boldsymbol{Y}}[\|\widehat{u}_{\phi}-u_{\lambda}^{*}\|_{H^{1}(\Omega)}^2]   \\
			\leq & \frac{2}{c_1\wedge 1} \left[C_{c_3}d(\mathcal{D}+3)(\mathcal{D}+2)\mathcal{W}\sqrt{\mathcal{D}+3+\log (d(\mathcal{D}+2)\mathcal{W})}\left(\frac{\log n}{n}\right)^{1/2} + \frac{c_3+1}{2}\varepsilon^2\right] \\
			& + \frac{2}{c_1\wedge 1} \left[C_{c_3}d\mathcal{D}\mathcal{W}\sqrt{\mathcal{D}+\log\mathcal{W}}\left(\frac{\log n}{n}\right)^{1/2} + \frac{C_{d}}{2}\varepsilon^2\right]\lambda  \\
			\leq & \frac{2}{c_1\wedge 1} \left[C_{c_3} 4d^2(\lceil\log d\rceil+6)(\lceil\log d\rceil+5)\left \lceil   \frac{Cc_2}{\varepsilon}-4\right \rceil^d\cdot\right.  \\
			& \quad\left.\sqrt{\lceil\log d\rceil+6+\log \left(4d^{2}(\lceil\log d\rceil+5)\left \lceil   \frac{Cc_2}{\varepsilon}-4\right \rceil^d\right)}\left(\frac{\log n}{n}\right)^{1/2} + \frac{c_3+1}{2}\varepsilon^2\right]    \\
			& +\frac{2}{c_1\wedge 1} \left[C_{c_3}4d^2(\lceil\log d\rceil+3)\left \lceil   \frac{Cc_2}{\varepsilon}-4\right \rceil^d\cdot\right. \\
			& \quad\left.\sqrt{\lceil\log d\rceil+3+\log \left(4d\left \lceil   \frac{Cc_2}{\varepsilon}-4\right \rceil^d\right)}\left(\frac{\log n}{n}\right)^{1/2} + \frac{C_{d}}{2}\varepsilon^2\right]\lambda  \\
			\leq & C_{c_1,c_2,c_3,d} {\mathcal{O}} \left( n^{-\frac{1}{d+2}}(\log n)^{\frac{d+3}{d+2}}\right)+C_{c_1,c_2,c_3,d} {\mathcal{O}} \left( n^{-\frac{1}{d+2}}(\log n)^{\frac{d+3}{d+2}}\right)\lambda.
		\end{aligned}
	\end{equation*}
	Using Theorem \ref{thm:errdec} and Theorem \ref{thm:penaltyerr}, it holds that for all $\lambda>0$
	\begin{equation}\label{eq:error:lambda}
		\begin{aligned}
			\mathbb{E}_{\boldsymbol{X},\boldsymbol{Y}}\left[\|\widehat{u}_{\phi}-u^{*}\|_{H^{1}(\Omega)}^2\right]
			\leq & C_{c_1,c_2,c_3,d} {\mathcal{O}} \left( n^{-\frac{1}{d+2}}(\log n)^{\frac{d+3}{d+2}}\right)         \\
			& +C_{c_1,c_2,c_3,d} {\mathcal{O}} \left( n^{-\frac{1}{d+2}}(\log n)^{\frac{d+3}{d+2}}\right)\lambda \\
			& +C_{c_3,d}\lambda^{-2}.
		\end{aligned}
	\end{equation}
	We have derive the error estimate for fixed $\lambda$, and now we are in a position to find a proper $\lambda$ and get the convergence rate. Since \eqref{eq:error:lambda} holds for any $\lambda>0$, we take the infimum of $\lambda$:
	\begin{equation*}
		\begin{aligned}
			\mathbb{E}_{\boldsymbol{X},\boldsymbol{Y}}\left[\|\widehat{u}_{\phi}-u^{*}\|_{H^{1}(\Omega)}^2\right]
			\leq \inf_{\lambda>0} & \left\{ C_{c_1,c_2,c_3,d} {\mathcal{O}} \left( n^{-\frac{1}{d+2}}(\log n)^{\frac{d+3}{d+2}} \right)\right.       \\
			& \left.+C_{c_1,c_2,c_3,d} {\mathcal{O}} \left( n^{-\frac{1}{d+2}}(\log n)^{\frac{d+3}{d+2}} \right)\lambda\right. \\
			& \left.+C_{c_3,d}\lambda^{-2}\right\}.
		\end{aligned}
	\end{equation*}
	By taking
	\begin{equation*}
		\lambda \sim n^{\frac{1}{3(d+2)}}(\log n)^{-\frac{d+3}{3(d+2)}},
	\end{equation*}
	we can obtain
	\begin{equation*}
		\mathbb{E}_{\boldsymbol{X},\boldsymbol{Y}}\left[\|\widehat{u}_{\phi}-u^{*}\|_{H^{1}(\Omega)}^2\right]
		\leq C_{c_1,c_2,c_3,d} {\mathcal{O}} \left( n^{-\frac{2}{3(d+2)}}\log n\right).
	\end{equation*}
\end{proof}

\section{Conclusions and Extensions}\label{section:conclusion}
\par This paper provided an analysis of convergence rate for deep Ritz methods for Laplace  equations with Dirichlet boundary condition. Specifically, our study  shed light on how to set depth and width of networks and how to set the penalty parameter to achieve the desired convergence rate in terms of number of training samples. The estimation on the approximation error of deep $\mathrm{ReLU}^{2}$ network is established in $H^1$. The statistical error can be derived technically by the Rademacher complexity of the non-Lipschitz composition of gradient norm and $\mathrm{ReLU}^{2}$ network. We also analysis the error from the boundary penalty method.
\par There are several interesting further research directions. First, the current analysis can be extended to general second order elliptic equations with other boundary conditions. Second, the approximation and statistical error bounds deriving here can be used for studying the nonasymptotic convergence rate for residual based method, such as PINNs. Finally, the similar result may be applicable to deep Ritz methods for optimal control problems and inverse problems.

\begin{acknowledgements}
	Y. Jiao is supported in part by the National Science Foundation of China under Grant 11871474 and by the research fund of KLATASDSMOE of China.  X. Lu is partially supported by the National Science Foundation of China (No. 11871385), the National Key Research and Development Program of China (No.2018YFC1314600) and the Natural Science Foundation of Hubei Province (No. 2019CFA007), and by the research fund of KLATASDSMOE of China.  J.  Yang was supported by NSFC (Grant No. 12125103, 12071362), the National Key Research and Development Program of China (No. 2020YFA0714200) and the Natural Science Foundation of Hubei Province (No. 2019CFA007).
\end{acknowledgements}

\appendix
\section{Appendix}
\subsection{Proof of Lemma \ref{lemma:regularity}}\label{app:lemma:regularity}
We claim that $a_\lambda$ is coercive on $H^1(\Omega)$. In fact,
\begin{equation*}
	a_\lambda(u,u)=a(u,u)+\lambda\int_{\partial\Omega}u^2\ \mathrm{d}s\geq C\|u\|^2_{H^1(\Omega)},\ \forall u\in H^1(\Omega),
\end{equation*}
where $C$ is constant from Poincar\'{e} inequality \cite{Gilbarg1983Elliptic}. Thus, there exists a unique weak solution $u_\lambda^*\in H^1(\Omega)$ such that \begin{equation*}
	a_\lambda(u_\lambda^*,v)=f(v), \ \forall v\in H^1(\Omega).
\end{equation*}
We can check $u_\lambda^*$ is the unique minimizer of $\mathcal{L}_{\lambda}(u)$ by standard technique.
\par We will study the regularity for weak solutions of \eqref{eq:robin}. For the following discussion, we first introduce several useful classic results of second order elliptic equations in \cite{Evans2010PartialDE,Gilbarg1983Elliptic}.
\begin{lemma}\label{lemma:regularity:dirichlet_f}
	Assume $w\in L^{\infty}(\Omega)$, $f\in L^{2}(\Omega)$, $g\in H^{3/2}(\partial\Omega)$ and $\partial \Omega$ is sufficiently smooth. Suppose that $u \in H^{1}(\Omega)$ is a weak solution of the elliptic boundary-value problem
	\begin{equation*}
		\left\{
		\begin{aligned}
			-\Delta u+wu & =f &  & \text { in } \Omega           \\
			u            & =g &  & \text { on } \partial \Omega.
		\end{aligned}
		\right.
	\end{equation*}
	Then $u \in H^{2}(\Omega)$ and there exists a positive constant $C$, depending only on $\Omega$ and $w$, such that
	\begin{equation*}
		\|u\|_{H^{2}(\Omega)} \leq C\left( \|f\|_{L^{2}(\Omega)}+\|g\|_{H^{3/2}(\partial\Omega)} \right).
	\end{equation*}
\end{lemma}

\begin{proof}
	See \cite{Evans2010PartialDE}.
\end{proof}

\begin{lemma}\label{lemma:regularity:neumann_g}
	Assume $w\in L^{\infty}(\Omega)$, $f\in L^{2}(\Omega)$, $g\in H^{1/2}(\partial\Omega)$ and $\partial \Omega$ is sufficiently smooth. Suppose that $u \in H^{1}(\Omega)$ is a weak solution of the elliptic boundary-value problem
	\begin{equation*}
		\left\{
		\begin{aligned}
			-\Delta u+wu                  & =f &  & \text { in } \Omega           \\
			\frac{\partial u}{\partial n} & =g &  & \text { on } \partial \Omega.
		\end{aligned}
		\right.
	\end{equation*}
	Then $u \in H^{2}(\Omega)$ and there exists a positive constant $C$, depending only on $\Omega$ and $w$, such that
	\begin{equation*}
		\|u\|_{H^{2}(\Omega)} \leq C\left( \|f\|_{L^{2}(\Omega)}+\|g\|_{H^{1/2}(\partial\Omega)} \right).
	\end{equation*}
\end{lemma}

\begin{proof}
	See \cite{Gilbarg1983Elliptic}.
\end{proof}

\begin{lemma}\label{lemma:regularity:robin}
	Assume $w\in L^{\infty}(\Omega)$, $g\in H^{1/2}(\partial\Omega)$, $\partial \Omega$ is sufficiently smooth and $\lambda>0$. Let $u\in H^1(\Omega)$ be the weak solution of the following Robin problem
	\begin{equation}\label{eq:regularity:robin}
		\left\{\begin{aligned}
			-\Delta u + wu                                    & =0 &  & \text { in } \Omega          \\
			\frac{1}{\lambda} \frac{\partial u}{\partial n}+u & =g &  & \text { on } \partial\Omega.
		\end{aligned}\right.
	\end{equation}
	Then $u\in H^{2}(\Omega)$ and there exists a positive constant $C$ independent of $\lambda$ such that
	\begin{equation*}
		\left\|u\right\|_{H^{2}(\Omega)} \leq C \lambda \|g\|_{H^{1/2}(\partial\Omega)}.
	\end{equation*}
\end{lemma}

\begin{proof}
	Following the idea which is proposed in \cite{Costabel1996ASP} in a slightly different context. We first estimate the trace $Tu=\left.u\right|_{\partial\Omega}$. We define the Dirichlet-to-Neumann map
	\begin{equation*}
		\widetilde{T}:\left.u\right|_{\partial\Omega} \mapsto \left.\frac{\partial u}{\partial n}\right|_{\partial\Omega},
	\end{equation*}
	where $u$ satisfies $-\Delta u+wu=0$ in $\Omega$, then
	\begin{equation*}
		Tu=\left(\frac{1}{\lambda}\widetilde{T}+I \right)^{-1}g.
	\end{equation*}
	Now we are going to show that $\frac{1}{\lambda}\widetilde{T}+I$ is a positive definite operator in $L^{2}(\partial\Omega)$. We notice that the variational formulation of \eqref{eq:regularity:robin} can be read as follow:
	\begin{equation*}
		\int_{\Omega}\nabla u\cdot\nabla vdx+\int_{\Omega}wuvdx+\lambda\int_{\partial\Omega}uvds=\lambda\int_{\partial\Omega}gvds, \ \forall v\in H^{1}(\Omega).
	\end{equation*}
	Taking $v=u$, then we have
	\begin{equation*}
		\| Tu \|_{L^{2}(\partial\Omega)}^{2} \leq \left\langle \left(\frac{1}{\lambda}\widetilde{T}+I \right)Tu,Tu \right\rangle.
	\end{equation*}
	This means that $\lambda^{-1}\widetilde{T}+I$ is a positive definite operator in $L^{2}(\partial\Omega)$, and further, $(\lambda^{-1}\widetilde{T}+I)^{-1}$ is bounded. We have the estimate
	\begin{equation}\label{eq:trace:estimate}
		\| Tu \|_{H^{1/2}(\partial\Omega)} \leq C\| g \|_{H^{1/2}(\partial\Omega)}.
	\end{equation}
	We rewrite the Robin problem \eqref{eq:regularity:robin} as follows
	\begin{equation*}
		\left\{\begin{aligned}
			-\Delta u + wu                  & =0                                                    &  & \text { in } \Omega          \\
			\frac{\partial u}{\partial n}+u & = \lambda \left(g-\left(1-\lambda^{-1}\right)u\right) &  & \text { on } \partial\Omega.
		\end{aligned}\right.
	\end{equation*}
	By Lemma \ref{lemma:regularity:neumann_g} we have
	\begin{equation}\label{eq:neumann:estimate}
		\| u \|_{H^{2}(\Omega)} \leq C\lambda \left\| g-\left(1-\lambda^{-1}\right)Tu \right\|_{H^{1/2}(\partial\Omega)} \leq C\lambda \left( \left\| g \right\|_{H^{1/2}(\partial\Omega)} + \left\| Tu \right\|_{H^{1/2}(\partial\Omega)} \right).
	\end{equation}
	Combining \eqref{eq:trace:estimate} and \eqref{eq:neumann:estimate}, we obtain the desired estimation.
\end{proof}
With the help of above lemmas, we now turn to proof the regularity properties of the weak solution.
\begin{theorem}
	Assume $w\in L^{\infty}(\Omega)$, $f \in L^{2}(\Omega)$. Suppose that $u \in H^{1}(\Omega)$ is a weak solution of the boundary-value problem \eqref{eq:robin}. If $\partial \Omega$ is sufficiently smooth, then $u \in H^{2}(\Omega)$, and we have the estimate
	\begin{equation*}
		\|u\|_{H^{2}(\Omega)} \leq C\|f\|_{L^{2}(\Omega)},
	\end{equation*}
	where the constant $C$ depending only on $\Omega$ and $w$.
\end{theorem}

\begin{proof}
	\par We decompose \eqref{eq:robin} into two equations
	\begin{equation}\label{eq:robin:u0}
		\left\{ \begin{aligned}
			-\Delta u_{0}+wu_{0} & =f &  & \text{in}\ \Omega          \\
			u_{0}                & =0 &  & \text{on}\ \partial\Omega, \\
		\end{aligned}\right.
	\end{equation}
	\begin{equation}\label{eq:robin:u1}
		\left\{ \begin{aligned}
			-\Delta u_{1}+wu_{1}                                     & =0                                  &  & \text{in}\ \Omega          \\
			\frac{1}{\lambda}\frac{\partial u_{1}}{\partial n}+u_{1} & =-\frac{\partial u_{0}}{\partial n} &  & \text{on}\ \partial\Omega. \\
		\end{aligned}\right.
	\end{equation}
	and obtain the solution of \eqref{eq:robin}
	\begin{equation*}
		u=u_{0}+\frac{1}{\lambda}u_{1}.
	\end{equation*}
	Applying Lemma \ref{lemma:regularity:dirichlet_f} to \eqref{eq:robin:u0}, we have
	\begin{equation}\label{eq:reg:u0}
		\|u_{0}\|_{H^{2}(\Omega)} \leq C \|f\|_{L^{2}(\Omega)},
	\end{equation}
	where $C$ depends on $\Omega$ and $w$. Using Lemma \ref{lemma:regularity:robin}, it is easy to obtain
	\begin{equation}\label{eq:reg:u1}
		\left\|u_{1}\right\|_{H^{2}(\Omega)} \leq C\lambda\left\|\frac{\partial u_{0}}{\partial n}\right\|_{H^{1/2}(\partial\Omega)}
		\leq  C\lambda\|u_{0}\|_{H^{2}(\Omega)},
	\end{equation}
	where the last inequality follows from the trace theorem. Combining \eqref{eq:reg:u0} and \eqref{eq:reg:u1}, the desired estimation can be derived by triangle inequality.
\end{proof}

%
%

\bibliographystyle{spmpsci}      
\bibliography{deep_ritz_ref}   

\end{document}